\documentclass{article}
\usepackage{amsmath,amssymb,amsthm}
\usepackage{enumerate}
\usepackage{color}
\usepackage{hyperref}
\usepackage[utf8]{inputenc}
\usepackage[a4paper,width={15cm},left=3cm,bottom=3.5cm, top=3.5cm]{geometry}
\usepackage{pgf,tikz}

\newcommand{\RR}{\ensuremath{\mathbb{R}}}
\newcommand{\CC}{\ensuremath{\mathbb{C}}}
\newcommand{\NN}{\ensuremath{\mathbb{N}}}
\newcommand{\eps}{\ensuremath{\varepsilon}}
\newcommand{\abs}[1]{\left\vert #1 \right\vert}
\newcommand{\nor}[1]{\left\Vert #1 \right\Vert}
\newcommand{\vA}{\ensuremath{\mathbf{A}}}
\newcommand{\va}{\ensuremath{\mathbf{a}}}

\newtheorem{thm}{Theorem}[section]
\newtheorem{prop}[thm]{Proposition}
\newtheorem{cor}[thm]{Corollary}
\newtheorem{lem}[thm]{Lemma}
\newtheorem{rem}[thm]{Remark}

\title{Eigenvalue  variation under moving mixed
Dirichlet--Neumann boundary conditions and applications}
\author{L. Abatangelo\footnote{
Dipartimento di Matematica e Applicazioni,
 Universit\`a di Milano--Bicocca,
Via Cozzi 55, 20125 Milano, Italy,
\texttt{laura.abatangelo@unimib.it}}, V. Felli\footnote{
Dipartimento di Scienza dei Materiali,
 Universit\`a di Milano--Bicocca,
Via Cozzi 55, 20125 Milano, Italy,
\texttt{veronica.felli@unimib.it}},
C. L\'ena\footnote{
Matematiska institutionen, Stockholms universitet, 106 91 Stockholm, Sweden, \texttt{corentin@math.su.se}
}}

\date{Revised version,
 March 27, 2019}

\begin{document}

\maketitle

\begin{abstract}
  We deal with the sharp asymptotic behaviour of eigenvalues of
  elliptic operators with varying mixed Dirichlet--Neumann boundary
  conditions.  In case of simple eigenvalues, we compute explicitly
  the constant appearing in front of the expansion's leading term.
  This allows inferring some remarkable consequences for
  Aharonov--Bohm eigenvalues when the singular part of the operator
  has two coalescing poles.
\end{abstract}



\section{Introduction and main results}

The present paper deals with elliptic  operators with varying mixed
Dirichlet--Neumann boundary conditions and their spectral stability
under varying of the Dirichlet and Neumann boundary regions. More
precisely, we study the behaviour of eigenvalues under a homogeneous Neumann condition on a
portion of the boundary concentrating at 
 a point  and  a homogeneous Dirichlet boundary condition on
the complement.

Let $\Omega$ be a bounded
open set in $\RR^2_+:=\{(x_1,x_2)\in\RR^2:x_2>0\}$ having the following properties:
\begin{align}\label{eq:38}
&\text{$\Omega$ is Lipschitz},\\
\label{eq:40}
&\text{there exists $\eps_0>0$ such that $\Gamma_{\eps_0}:=[-\eps_0,\eps_0]\times\{0\}\subset \partial \Omega$}.
\end{align}
We consider the eigenvalue problem for the
Dirichlet Laplacian on the  domain $\Omega$
\begin{equation}\label{eqD}
\begin{cases}
  -\Delta\,u=\lambda\,u,& \mbox{in }\Omega,\\
  u=0,&\mbox{on } \partial\Omega.
         \end{cases}
\end{equation}
	We denote by $(\lambda_j)_{j\ge 1}$  the eigenvalues of Problem \eqref{eqD}, arranged in non-decreasing order and counted with multiplicities.
	
	For each $\eps\in (0,\eps_0]$, we also consider the following eigenvalue problem with mixed boundary conditions:
\begin{equation}\label{eqDND}
\begin{cases}
-\Delta\,u=\lambda\,u,& \mbox{in }\Omega,\\
u=0,&\mbox{on } \partial\Omega\setminus \Gamma_{\eps},\\
\frac{\partial u}{\partial \nu}=0,&\mbox{on } \Gamma_{\eps},
\end{cases}
\end{equation}
	with $\Gamma_{\eps}:=[-\eps,\eps]\times\{0\}$, see Figure \ref{fig:f1}.
	We denote by $(\lambda_j(\eps))_{j\ge 1}$  the eigenvalues of Problem \eqref{eqDND}, arranged in non-decreasing order and counted with multiplicities.

	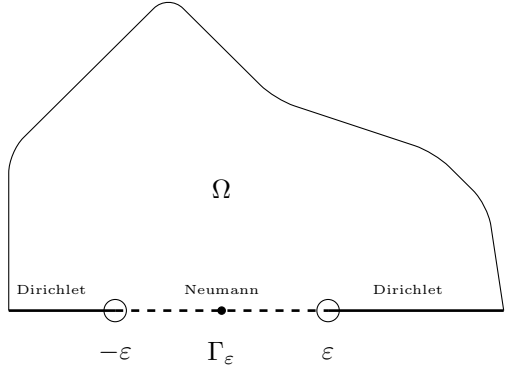
\begin{figure}\label{fig:f1}
	\centering
	\begin{tikzpicture}[scale=0.7]
 \draw[rounded corners=1.8ex]
 (0,0)  --
(0,3)  --
(3,6)  --
(5,4)  --
(8,3)  --
(9,2)  --
(9.3,0); 
 \draw[dashed,line width=1pt](2,0) -- (6,0);
  \draw (4,0.4) node {\tiny Neumann};
  \draw[line width=1pt](0,0) -- (2,0);
 \draw (0.8,0.4) node {\tiny Dirichlet};
\draw[line width=1pt](6,0) -- (9.3,0);
 \draw (7.5,0.4) node {\tiny Dirichlet};
\draw  (2,0) circle (6pt);
\draw[fill=black]  (4,0) circle (2pt);
\draw (2,-0.8) node {$-\eps$};
 \draw  (6,0) circle (6pt);
 \draw (6,-0.8) node {$\eps$};
 \draw (4,-0.8) node {$\Gamma_\eps$};
\draw (4,2.3) node {$\Omega$};
\end{tikzpicture}
\caption{The mixed boundary condition problem in the  domain $\Omega$.}
\end{figure}
	
A rigorous weak formulation of the eigenvalue problems described above
can be given as follows. For $\eps \in (0,\eps_0]$, we define
\begin{equation*}
  \mathcal{Q}_{\eps}=\left\{u\in H^1(\Omega):\,\chi_{\partial\Omega\setminus\Gamma_{\eps}}\gamma_0(u)=0\mbox{ in } L^2(\partial \Omega)\right\},
\end{equation*}
where $\gamma_0$ is the trace operator from $H^1(\Omega)$ to
$L^2(\partial \Omega)$, which is a continuous linear mapping (see for
instance \cite[Definition 13.2]{Tartar2007Sobolev}) and
$\chi_{\partial\Omega\setminus\Gamma_{\eps}}$ is the indicator function of $\partial\Omega\setminus\Gamma_{\eps}$ in
$\partial \Omega$. Furthermore, we define the quadratic form $q$ on
$H^1(\Omega)$ by
\begin{equation}
		\label{eqQuad}
		q(u):=\int_{\Omega}\left|\nabla u\right|^2\,dx.
\end{equation}
Let us denote by $q_0$ the restriction of $q$ to $H^1_0(\Omega)$ and
by $q_{\eps}$ the restriction of $q$ to $\mathcal Q_{\eps}$.  The
sequences $(\lambda_j)_{j\ge 1}$ and $(\lambda_j(\eps))_{j\ge 1}$ for
$\eps\in (0,\eps_0]$ can then be defined by the min-max principle:
\begin{equation}\label{eqMinMaxD}
  \lambda_j:=\min_{\substack{
\mathcal E\subset H^1_0(\Omega)\text{ subspace}\\
\mbox{dim}(\mathcal E)=j}}\max_{u\in \mathcal E} \frac{q(u)}{\|u\|^2}
\end{equation}
	and 
\begin{equation}\label{eqMinMaxDND}
  \lambda_j(\eps):=\min_{\substack{\mathcal E\subset
      \mathcal{Q}_{\eps}\text{ subspace}\\
\mbox{dim}(\mathcal E)=j}}\max_{u\in \mathcal E} \frac{q(u)}{\|u\|^2},
\end{equation}
where 
\[
\|u\|^2=\int_{\Omega}u^2(x)\,dx.
\]
Since $H^1(\Omega)$ is compactly embedded in $L^2(\Omega)$ (see
e.g. \cite[Lemma 18.4]{Tartar2007Sobolev}), the eigenvalues of $q_0$,
defined by Equation \eqref{eqMinMaxD}, and those of $q_{\eps}$, defined
by Equation \eqref{eqMinMaxDND}, are of finite multiplicity, and form sequences tending to $+\infty$.


\begin{rem} \label{remIneq}
 Let us fix $\eps_1$ and $\eps_2$ in
  $(0,\eps_0]$ such that $\eps_1>\eps_2$. Since
  $H^1_0(\Omega)\subset\mathcal Q_{\eps_2} \subset\mathcal Q_{\eps_1}$, the
  definitions given by Formulas \eqref{eqMinMaxD} and
  \eqref{eqMinMaxDND} immediately imply that
  $\lambda_j(\eps_1)\le \lambda_j(\eps_2)\le \lambda_j$ for each
  integer $j\ge 1$. The function
  $(0,\eps_0]\ni \eps\mapsto \lambda_j(\eps)$ is therefore
  non-increasing and bounded by $\lambda_j$ for each integer $j\ge1$.
\end{rem}

\begin{rem} \label{rem:conf_change}
For the sake of simplicity, in the present paper we assume that the
domain $\Omega$ satisfies assumption \eqref{eq:40}, i.e. that $\partial\Omega$
is straight in a neighborhood of $0$. We observe that, since we are in
dimension $2$, 
this assumption is not restrictive. Indeed, starting from a general
sufficiently regular domain $\Omega$, a conformal
transformation leads us to consider a new domain satisfying
\eqref{eq:40}, see e.g. \cite{FFFN}. The counterpart is the appearance of a
conformal weight in the new eigenvalue problem; however, if $\Omega$
is sufficiently regular, the
weighted problem presents no additional difficulties.
\end{rem}

The purpose of the present paper is to study the
  eigenvalue function  $\eps\mapsto\lambda_j(\eps)$ as
$\eps\to 0^+$. The continuity of this map as well as some asymptotic
expansions were obtained in \cite{Gad} (see also Appendix C of the
present paper for an alternative proof of some results of
\cite{Gad}).  Here we mean to provide some
explicit characterization of the leading terms in the expansion given
in \cite{Gad} and of the limit profiles arising from blowing-up of eigenfunctions.

Spectral stability and asymptotic expansion of the
  eigenvalue variation in a somehow complementary setting were
  obtained in \cite{AFHL2016}; indeed, if we consider  the eigenvalue problem under
  homogeneous Dirichlet boundary conditions on a
vanishing portion of a straight part of the  boundary, Neumann
conditions on
the complement in the straight part and Dirichlet conditions elsewhere, by
a reflection the problem becomes
equivalent to the one studied in \cite{AFHL2016}, i.e. a Dirichlet eigenvalue problem in a domain with a small
segment removed.

Related spectral stability results were discussed in
\cite[Section 4]{ColoradoPeral2003} for the first eigenvalue under mixed
Dirichlet-Neumann boundary conditions on a smooth bounded domain
$\Omega\subset\RR^N$ ($N\geq3$), both for  vanishing
Dirichlet boundary portion  and for vanishing Neumann boundary portion.

We also mention that some regularity results for solutions to
  second-order elliptic problems with mixed Dirichlet--Neumann type
  boundary conditions were obtained in \cite{Kassmann,Savare}, see
  also the references therein, whereas asymptotic expansions at
  Dirichlet-Neumann boundary  junctions
were derived in \cite{FFFN}.

Let us assume that 
\begin{equation}\label{eq:6}
\lambda_N \text{ (i.e. the $N$-th eigenvalue of $q_0$) is simple}.
\end{equation}
Let $u_N$ be a normalized eigenfunction associated to $\lambda_N$, i.e. $u_N$ satisfies 
  \begin{equation}\label{eq:5}
   \begin{cases}
 -\Delta u_N=\lambda_N u_N,&\text{in }\Omega,\\
u_N=0,&\text{on }\partial\Omega,\\
\int_\Omega u_N^2(x)\,dx=1.
\end{cases}
\end{equation}
It is known (see \cite{Gad}) that, under assumption \eqref{eq:6}, the
 rate of the convergence $\lambda_N^\eps\to\lambda_N$ is strongly related to 
the order of vanishing of the Dirichlet
eigenfunction $u_N$ at $0$.
Moreover $u_N$ has an integer order of vanishing $k\geq1$
at $0\in\partial\Omega$ and there exists
  $\beta\in\RR\setminus\{0\}$ such that 
\begin{align}\label{eq:orderk}
r^{-k} u_N(r\cos t,r\sin t ) \to \beta
\sin(kt) \text{ as }r\to 0 \text{ in }C^{1,\tau}([0,\pi]) 
\end{align}
for any $\tau\in(0,1)$, see e.g. \cite[Theorem 1.1]{FF}.
Actually, one can see that $\beta$ is directly linked to the norm of the $k$-th differential of $u_N$ at $0$. More precisely, if we consider 
\[
 \| d^{j}u (x) \|^2 := \sum_{i_1,\ldots,i_j=1}^2 \left| \frac{\partial ^j u}{\partial x_{i_1}\ldots \partial x_{i_j}}(x) \right|^2 ,
\]
then 
\[
\beta^2 = \frac{\| d^{k}u_N (0) \|^2}{(k!)^2 \,2^{k-1}}.
\]
 This follows by differentiating the harmonic homogeneous functions $\beta r^k \sin(kt)$ and $\beta r^k \cos(kt)$ with respect to $x_1$ and $x_2$ and considering $d^{k}u_N (0)$.

Our main results provide sharp
asymptotic estimates with explicit coefficients for the eigenvalue variation
$\lambda_N-\lambda_N(\eps)$ as $\eps\to0^+$ under assumption
\eqref{eq:6} (Theorem \ref{t:new_main}),
as well as an explicit representation in elliptic
  coordinates of the limit blow-up
  profile for the corresponding eigenfunction $u_N^\eps$
  (Theorem \ref{t:blowup2}).

\begin{thm}\label{t:new_main}
Let $\Omega$ be a bounded open set in $\RR^2$ satisfying 
\eqref{eq:38} and \eqref{eq:40}.
Let $N\geq 1$ be such that 
the $N$-th eigenvalue $\lambda_N$ of $q_0$ on $\Omega$  is simple 
 with associated eigenfunctions having  in
  $0$ a zero of order $k$ with $k$ as in \eqref{eq:orderk}. For
  $\eps\in (0,\eps_0)$, let $\lambda_N(\eps)$ be
  the $N$-th eigenvalue of $q_\eps$ on $\Omega$. 
Then 
\begin{equation*}
\lim_{\eps\to 0^+}\frac{\lambda_N-\lambda_N(\eps)}{\eps^{2k}}=
 \beta^2\, \frac{k\pi}{2^{2k-1}}
\binom{k-1}{\left\lfloor\frac{k-1}2\right\rfloor}^{\!2}
\end{equation*}
with $\beta\neq0$ being as in \eqref{eq:5}--\eqref{eq:orderk}.
\end{thm}

\begin{thm}\label{t:blowup2}
Let $\Omega$ be a bounded open set in $\RR^2$ satisfying 
\eqref{eq:38}  and \eqref{eq:40}.
Let $N\geq 1$ be such that 
the $N$-th eigenvalue $\lambda_N$ of $q_0$ on $\Omega$  is simple 
 with associated eigenfunctions having  in
  $0$ a zero of order $k$ with $k$ as in \eqref{eq:orderk}. For
  $\eps\in [0,\eps_0)$, let $\lambda_N(\eps)$ be
  the $N$-th eigenvalue of $q_\eps$ on $\Omega$
and $u_N^{\eps}$ be an
associated eigenfunction satisfying $\int_\Omega |u_N^{\eps}|^2\,dx=1$
and $\int_{\Omega}u_N^{\eps}\,u_N \,dx\ge0$. Then 
  \[
\eps^{-k}u_N^\eps(\eps x)\to
\beta (\psi_k+W_k\circ F^{-1})\quad\text{as
}\eps\to0^+
\]
in $H^{1}_{\rm loc}(\overline{\RR^2_+})$, a.e. and in 
$C^{2}_{\rm loc}(\overline{\RR^2_+}\setminus\{(1,0),(-1,0)\})$,
 where $\beta$ is as in \eqref{eq:5}--\eqref{eq:orderk},
\begin{align}\label{eq:psi_k}
& \psi_k(r\cos t,r\sin t)=r^k \sin(kt), \quad\text{for }t\in[0,\pi] \text{ and }r>0,\\
\label{eq:44}
&F(\xi,\eta)=(\cosh(\xi)\cos(\eta),\sinh(\xi)\sin(\eta)),\quad
  \text{for }\xi\geq0,\ \eta\in[0,2\pi),
\end{align}
 and 
 \begin{equation}\label{eq:W_k}
    W_k(\xi,\eta)=\frac1{2^{k-1}}\sum_{j=0}^{\left\lfloor{\frac{k-1}2}\right\rfloor}
\binom kj
\exp(-(k-2j)\xi)\sin((k-2j)\eta).
 \end{equation}
\end{thm}

Actually, 
the fact that $\lim_{\eps\to
  0^+}\frac{\lambda_N-\lambda_N(\eps)}{\eps^{2k}}$ is finite and
different from zero and the convergence of $\eps^{-k}u_N^\eps(\eps x)$
to some nontrivial profile was proved in the paper \cite{Gad}
with a
quite  implicit description of the limits (see also Appendix
\ref{sec:altern-proof-theor} for an alternative proof).
The original contribution of the present paper relies in the explicit 
characterization of the leading term of the expansion provided by \cite{Gad}
and in its  applications to Aharonov--Bohm operators, see
Section \ref{subs:IntroAB}.
The key tool  allowing us to write explicitly the coefficients of the
expansion consists in the use of elliptic coordinates, which turn out
to be  more suitable  
to our problem than
radial ones, see Section \ref{sec:explicit}.

\section{Applications to Aharonov--Bohm operators}
\label{subs:IntroAB}

The present work is in part motivated by the study of Aharonov--Bohm
eigenvalues.
In this section we describe some applications of Theorem
\ref{t:new_main} to the problem of spectral stability of
Aharonov--Bohm operators with two moving poles, referring to Section
\ref{s:AB} for the proofs.

 Let us first review  some definitions and known
results. For any point $\va=(a_1,a_2)\in \RR^2$, we define the
Aharonov--Bohm potential of circulation $1/2$ by
\begin{equation*}
	\vA_\va(x)=\frac12\left(\frac{-(x_2-a_2)}{(x_1-a_1)^2+(x_2-a_2)^2},\frac{x_1-a_1}{(x_1-a_1)^2+(x_2-a_2)^2}\right).
\end{equation*}
Let us consider an open and bounded open set $\widehat\Omega$ with Lipschitz boundary, such that $0\in\widehat\Omega$. For better readability, we denote by $\mathcal H$ the complex Hilbert space of complex-valued functions $L^2(\widehat\Omega,\CC)$, equipped with the scalar product defined, for all $u,v\in \mathcal H$, by 
\begin{equation*}
	\langle u,v\rangle:=\int_{\widehat\Omega}u\overline v\,dx.
\end{equation*}
We define, for $\va\in \widehat\Omega$,
\begin{equation}
\label{eq:QAB1Pole}
	\mathcal Q^{AB}_\va:=\left\{u\in  H^1_0\big(\widehat{\Omega},\CC\big)\,;\, \frac{|u|}{|x-\va|}\in L^2\big(\widehat{\Omega}\big)\right\},
\end{equation}
the quadratic form $q^{AB}_\va$ on $\mathcal Q^{AB}_\va$ by
\begin{equation}
\label{eq:QuadAB1Pole}
q^{AB}_\va(u):=\int_{\widehat\Omega}\left|(i\nabla+\vA_\va)u\right|^2\,dx,
\end{equation}
and the sequence of eigenvalues $\left(\lambda^{AB}_j(\va)\right)_{j\ge 1}$ by the min-max principle
\begin{equation}
\label{eq:MinMaxAB1Poles}
	\lambda_j^{AB}(\va):=  \min_{\substack{\mathcal E\subset
      \mathcal{Q}^{AB}_{\va}\text{ subspace}\\
\mbox{dim}(\mathcal E)=j}}\max_{u\in \mathcal E} \frac{q^{AB}_\va(u)}{\|u\|^2}.
\end{equation}
It follows from the definition in Equation \eqref{eq:QAB1Pole} that $\mathcal Q^{AB}_\va$ is compactly embedded in $\mathcal H$. The above eigenvalues are therefore of finite multiplicity and $\lambda_j^{AB}(\va)\to +\infty$ as $j\to +\infty$.

\begin{rem} Let us note that, as shown in \cite[Lemma 2.1]{NT},
    $\mathcal Q^{AB}_\va$ is the completion of the set of smooth functions supported in $\widehat\Omega\setminus\{\va\}$ for the norm $\|\cdot\|_\va$ defined by
\begin{equation*}
	\|u\|_\va^2=\|u\|^2+q^{AB}_\va(u).
\end{equation*}
Let us point out that functions in $\mathcal Q^{AB}_\va$ satisfy a Dirichlet boundary condition,  which is not the case in \cite{NT}. However, this difference is unimportant for the compact embedding.
\end{rem}

\begin{rem} We could also consider the Friedrichs extension of the differential operator 
\begin{equation*}
(i\nabla+\vA_\va)^*(i\nabla+\vA_\va) u=	-\Delta u+2i\vA_\va\cdot\nabla u+\left|\vA_\va\right|^2u
\end{equation*} 
acting on functions $u\in C^{\infty}_c(\widehat\Omega\setminus\{\va\},\CC)$. 
As shown for instance in \cite[Section I]{Len15AB} or \cite[Section 2]{BNNT}), this defines a positive and self-adjoint operator with compact resolvent, whose eigenvalues, counted with multiplicities, are $\left(\lambda^{AB}_j(\va)\right)_{j\ge 1}$. It is called the Aharonov-Bohm operator of pole $\va$ and circulation $1/2$. 
\end{rem}

In recent years, several authors have studied the dependence of
eigenvalues on the position of the pole.  It has been established in
\cite[Theorem 1.1]{BNNT}, that the functions
$\va \mapsto \lambda^{AB}_j(\va)$ are continuous in
$\overline{\Omega}$. In
\cite{abatangelo2015sharp,AbatangeloFelli2016SIAM}, the two first
authors obtained the precise rate of convergence $\lambda^{AB}_j(\va)
\to \lambda^{AB}_j(0)$ as $\va$ converges to the interior
point $0$
for  simple eigenvalues. In order to state the most complete result,
given in \cite[Theorem 1.2]{AbatangeloFelli2016SIAM}, we consider an
$L^2$-normalized eigenfunction $u_N^0$ of $q_0^{AB}$ associated with
the eigenvalue $\lambda^{AB}_N(0)$. We additionally assume that
$\lambda^{AB}_N(0)$ is simple. From \cite[Section 7]{FFT} it follows
that there exists an odd positive integer $k$ and a non-zero complex
number $\beta_0$ such that, up to a rotation of the
coordinate axes,
\begin{equation*}
	r^{-\frac{k}2}u_N^0(r\cos t,r\sin t)\to {\beta_0} e^{i\frac{t}2}\sin\left(\frac{k}2t\right)\mbox{ in }C^{1,\tau}\left([0,2\pi],\CC\right)
\end{equation*} 
as $r \to 0^+$, for all $\tau\in (0,1)$. The integer $k$ has a simple
geometric interpretation: it is the number of nodal lines of the
function $u_N^0$ which meet at $0$. We say that $u_N^0$
has a zero of order $k/2$ in $0$. Our coordinate axes are chosen in
such a way that one of these nodal lines is tangent to the
positive $x_1$ semi-axis.

\begin{thm}	\label{t:monopole} Let us define $\va_\eps:=\eps(\cos(\alpha),\sin(\alpha))$, with $\eps>0$. We have, as $\eps\to 0^+$,
\begin{equation*}
	\lambda^{AB}_N\left(\va_\eps\right)=\lambda^{AB}_N(0)-\frac{k\pi\beta_0^2}{2^{2k-1}}\left(\begin{array}{c} k-1\\ \left\lfloor \frac{k-1}{2}\right\rfloor\end{array}\right)^2\cos(k\alpha)\eps^k+o\left(\eps^k\right).
\end{equation*}
\end{thm}

\begin{rem} The expansion in \cite{abatangelo2015sharp,AbatangeloFelli2016SIAM} involves a constant depending on $k$, defined as the minimal energy in a Dirichlet-type problem. We compute this constant in Appendix \ref{a:A} in order to obtain the more explicit result in Theorem \ref{t:monopole}.
\end{rem}

Let us now consider, for any $\eps>0$, an Aharonov-Bohm potential with two poles $(\eps,0)$ and $(-\eps,0)$, of fluxes respectively $1/2$ and $-1/2$:
\begin{equation*}
	\vA_{\eps}:=\vA_{(\eps,0)}-\vA_{(-\eps,0)}.
\end{equation*}
As in the case of one pole, we define the vector space $\mathcal Q^{AB}_\eps$ by
\begin{equation}
\label{eq:QAB2Poles}
\mathcal Q^{AB}_\eps:=\left\{u\in  H^1_0\big(\widehat{\Omega},\CC \big)\,;\, \frac{|u|}{|x\pm\eps\,{\mathbf e}|}\in L^2\big(\widehat{\Omega}\big)\right\},
\end{equation}
 where $\mathbf e=(1,0)$, the quadratic form 
 $q^{AB}_\eps$ on $\mathcal Q^{AB}_\eps$ by
\begin{equation}
\label{eq:QuadAB2Poles}
q^{AB}_\eps(u):=\int_{\widehat\Omega}\left|(i\nabla+\vA_\eps)u\right|^2\,dx,
\end{equation}
and the sequence of eigenvalue $\left(\lambda^{AB}_j(\eps)\right)_{j\ge 1}$ by the min-max principle
\begin{equation}
\label{eq:MinMaxAB1Poles}
	\lambda_j^{AB}(\eps):=  \min_{\substack{\mathcal E\subset
      \mathcal{Q}^{AB}_{\eps}\text{ subspace}\\
\mbox{dim}(\mathcal E)=j}}\max_{u\in \mathcal E} \frac{q^{AB}_\eps(u)}{\|u\|^2}.
\end{equation}
It follows from \cite[Corollary 3.5]{Len15AB} that, for any $j\ge1$,
$\lambda^{AB}_j(\eps)$ converges to the $j$-th eigenvalue of the
Laplacian in $\widehat\Omega$ as $\eps\to 0^+$.  In
\cite{AFL2017ANS,AFHL2016} the authors obtained in some cases a sharp
rate of convergence. In order to state the result, let us
introduce some notation. We denote by $\widehat q_0$ the quadratic
form on $H^1_0(\widehat \Omega)$ defined by Equation \eqref{eqQuad},
replacing $\Omega$ with $\widehat\Omega$, and we denote by
$\big(\widehat\lambda_j\big)_{j\ge1}$ the sequence of eigenvalues
defined by Equation \eqref{eqMinMaxD},  replacing $\Omega$ with
$\widehat\Omega$ and $q$ with $\widehat q_0$. We fix an integer
$N\ge1$ and assume that $\widehat\lambda_N$ is a simple eigenvalue. We
denote by $\widehat u_N$ an associated eigenfunction, normalized in
$L^2\big(\widehat\Omega\big)$.
\begin{thm}{\cite[Theorem 1.2]{AFL2017ANS}} \label{thmNonZero} If $\widehat u_N(0)\neq 0$, we have, as $\eps\to0$,
\begin{equation*}
  \lambda_N^{AB}(\eps)=\widehat\lambda_N+\frac{2\pi}{|\log(\eps)|}\widehat u_N ^2 (0)+o\left(\frac1{|\log(\eps)|}\right).
\end{equation*}
\end{thm}

In the case $\widehat u_N(0)=0$, it is well known that there exist  $k\in\NN\setminus\{0\}$, $\widehat\beta\in \RR\setminus\{0\}$ and $\alpha\in [0,\pi)$ 
such that
\begin{equation*}
r^{-k}	\widehat u_N (r\cos t,r\sin t)\to \widehat\beta \sin\left(\alpha-kt\right)\mbox{ in }C^{1,\tau}\left([0,2\pi],\CC\right)
\end{equation*}
as $r\to 0^+$ for all $\tau\in (0,1)$. In particular, there is a nodal
line whose tangent makes the angle $\alpha/k$ with the 
  positive $x_1$ semi-axis.  As above we can characterize
  $\widehat\beta$ as $|\widehat\beta|^2 = \frac{\| d^{k}\widehat u_N (0) \|^2}{(k!)^2 \,2^{k-1}}$.\ 

Let us assume that 
\[ 
\widehat\Omega \text{ is symmetric with respect to the $x_1$-axis.} 
\] 
Since $\widehat\lambda_N$ is simple, $\widehat u_N$ is either even or odd in the variable $x_2$ and $\alpha$ is either $\pi/2$ or $0$  accordingly. 
\begin{thm}{\cite[Theorem 1.16]{AFHL2016}} \label{t:dipoleSymEven}
 If $\widehat u_N$ is even in $x_2$, which corresponds to $\alpha=\pi/2$, we have, as $\eps\to 0^+$,
\begin{equation*}
	\lambda_N^{AB}(\eps)=\widehat\lambda_N+\frac{k\pi{\widehat\beta}^2}{4^{k-1}}\left(\begin{array}{c} k-1\\ \left\lfloor\frac{k-1}2\right\rfloor\end{array}\right)^2\eps^{2k}+o\left(\eps^{2k}\right).
\end{equation*}
\end{thm}

\begin{rem} The statements in \cite{AFHL2016} contain a constant $C_k$ which we put in a simpler form in Appendix \ref{a:A}, in order to obtain Theorem \ref{t:dipoleSymEven}. \end{rem}

As a corollary of Theorem \ref{t:new_main}, we prove in Section \ref{s:AB} the following result, which complements the previous theorem.

\begin{thm} \label{t:dipoleSymOdd}
	If $\widehat u_N$ is odd in $x_2$, which corresponds to $\alpha=0$, we have, as $\eps\to 0^+$,
\begin{equation*}
	\lambda_N^{AB}(\eps)=\widehat\lambda_N-\frac{k\pi{\widehat\beta}^2}{4^{k-1}}\left(\begin{array}{c} k-1\\ \left\lfloor\frac{k-1}2\right\rfloor\end{array}\right)^2\eps^{2k}+o\left(\eps^{2k}\right).
\end{equation*}
\end{thm}

\begin{rem}	As discussed in Section \ref{s:AB}, the assumption
  that $\widehat\lambda_N$ is simple can be slightly relaxed,
admitting, in some cases, also double eigenvalues.
\end{rem}

\section{Sharp asymptotics for the eigenvalue
  variation}\label{sec:explicit}

  \subsection{Related results from the literature}\label{subsec:relres}
As already mentioned in the introduction, some
  asymptotic expansions for the eigenvalue variation
  $\lambda_N-\lambda_N(\eps)$ were derived in  \cite{Gad}. Let us first
  recall the results from \cite{Gad} which are the starting of our analysis.

Let $s:=\{(x_1,x_2)\in\RR^2: x_2=0\text{ and }x_1\geq 1\ \text{ or }x_1 \leq -1\}$.  We denote as $\mathcal Q$
the completion of $C^\infty_{\rm c}(\overline{\RR^2_+} \setminus s)$
under the norm $( \int_{\RR^2_+} |\nabla u|^2\,dx )^{1/2}$. 
  From the
Hardy type inequality proved in \cite{LW99} and
a change of gauge, it follows that
functions in $\mathcal Q$ 
satisfy the Hardy type inequalities
\begin{equation}\label{eq:1} 
\frac14 \int_{\RR^2_+} \dfrac{|\varphi(x)|^2}{|x-{\mathbf e}|^2}\,dx\leq
\int_{\RR^2_+} |\nabla \varphi(x)|^2\,dx,
\quad\text{for all }\varphi\in \mathcal Q,
\end{equation}
and 
\begin{equation}\label{eq:2}
\frac14 \int_{\RR^2_+} \dfrac{|\varphi(x)|^2}{|x+{\mathbf e}|^2}\,dx\leq
\int_{\RR^2_+} |\nabla \varphi(x)|^2\,dx,
\quad\text{for all }\varphi\in \mathcal Q,
\end{equation}
where $\mathbf e=(1,0)$. Inequalities \eqref{eq:1} and \eqref{eq:2}
allow characterizing $\mathcal Q$ as the following concrete functional
space:
\[
\mathcal Q=\Big\{ u\in L^1_{\rm loc}(\RR^2_+):\
 \nabla u\in L^2(\RR^2_+), 
\ \tfrac{u}{|x\pm\mathbf e|}\in L^2(\RR^2_+ ), \text{ and } u=0 \text{
  on }s\Big\}. 
\]

We refer to the paper \cite{Gad}, where the following 
  theorem can be found 
as a particular case of more general results.

\begin{thm}[\cite{Gad}]\label{t:gad}
 Let $\Omega$ be a bounded open set in $\RR^2$ satisfying 
\eqref{eq:38} and \eqref{eq:40}.
Let $N\geq 1$ be such that 
the $N$-th eigenvalue $\lambda_N$ of $q_0$ on $\Omega$  is simple 
 with associated eigenfunction $u_N$ having  in
  $0$ a zero of order $k$ with $k$ as in \eqref{eq:orderk}. For
  $\eps\in (0,\eps_0)$, let $\lambda_N(\eps)$ be
  the $N$-th eigenvalue of $q_\eps$ on $\Omega$ and $u_N^\eps$ be its 
 associated eigenfunction, normalized to satisfy  $\int_\Omega |u_N^{\eps}|^2\,dx=1$
and $\int_{\Omega}u_N^{\eps}\,u_N \,dx\ge0$. 
Then, as $\eps\to0^+$,
\begin{align}
&\frac{\lambda_N-\lambda_N(\eps)}{\eps^{2k}} \to -
 \beta^2\, \displaystyle{\int_{-1}^1 \frac{\partial w_k}{\partial x_2}\,w_k \, dx_1 }, \label{eq:gad}\\
 &\eps^{-k}u_N^\eps(\eps x) \to
\beta (\psi_k+w_k) \quad \text{in $H^{1}_{\rm loc}(\overline{\RR^2_+})$}, \label{eq:gad2}
\end{align}
with $\beta\neq0$ being as in \eqref{eq:5}--\eqref{eq:orderk},
$\psi_k$ being defined in \eqref{eq:psi_k}, and $w_k$  being the
unique $\mathcal Q$-weak solution to the problem
 \begin{equation}\label{eq:wk}
 \begin{cases}
   -\Delta w_k=0, &\text{in }\RR^2_+, \\
   w_k=0, &\text{on }s, \\
   \frac{\partial w_k}{\partial \nu}=-\frac{\partial \psi_k}{\partial
     \nu}, &\text{on }\Gamma_1.
 \end{cases}
\end{equation} 
\end{thm}

Convergence  \eqref{eq:gad} can be obtained combining \cite[Equation (4.6)]{Gad} 
for simple eigenvalues, \cite[Equation (3.4)]{Gad} together with \cite[Lemma 3.3]{Gad}.
As well, \eqref{eq:gad2} is given by \cite[Equation (2.3)]{Gad}, which is a consequence of 
\cite[Theorem 5.2]{Gad}, \cite[Equation (4.10)]{Gad}, \cite[Lemma 3.3]{Gad}.
For the sake of clarity and completeness, we present an alternative
proof in Appendix \ref{sec:altern-proof-theor}, which relies on
energy estimates obtained by an Almgren type monotonicity argument and
blow-up analysis.

We remark that in \cite{Gad} the author describes the 
  limit profile $w_k$ 
  solving \eqref{eq:wk}
with polar coordinates. On the contrary, our contribution relies essentially on 
  the use of elliptic coordinates 
in place of polar ones. 
  This allows us to compute explicitly the right hand side
  of \eqref{eq:gad}, thus obtaining the following result.

\begin{prop} \label{propmk} For any positive integer $k$, 
\begin{equation*}
	\int_{-1}^1 \frac{\partial w_k}{\partial x_2}\,w_k \, dx_1
	=-\frac{k\pi}{2^{2k-1}}\left(\begin{array}{c}k-1\\ \left\lfloor\frac{k-1}2\right\rfloor\end{array}\right)^{\!2}.
\end{equation*}
\end{prop}

The proof of Proposition \ref{propmk} relies in an
  explicit construction of the limit profile $w_k$, 
using a parametrization of the upper half-plane $\RR^2_+$ by elliptic coordinates, 
a finite trigonometric expansion, and the simplification of a sum involving binomial coefficients. 

\subsection{Computation of the limit profile $w_k$}

Let us first compute $w_k$. By uniqueness, any function
in the functional space $\mathcal Q$ 
that satisfies all the conditions of Problem \eqref{eq:wk} is equal to
$w_k$. In order to find such a function, we use the elliptic
coordinates $(\xi,\eta)$ defined by
\begin{equation}\label{eqChangeCoord}
\begin{cases}
x_1=\cosh(\xi)\cos(\eta),\\
x_2=\sinh(\xi)\sin(\eta).
\end{cases}
\end{equation}
More precisely, we consider the function
$F:(\xi,\eta)\mapsto (x_1,x_1)$ defined by the equations
\eqref{eqChangeCoord}. It is a $C^{\infty}$ diffeomorphism from
$D:=(0,+\infty)\times (0,\pi)$ to $\RR^2_+$. We note that $F$ is
actually a conformal mapping. Indeed, if we define the complex
variables $z:=x_1+ix_2$ and $\zeta:=\xi+i\eta$, we have
$z=\cosh(\zeta)$, which proves the claim since $\cosh$ is an entire
function. Let us denote by $h(\xi,\eta)$ the scale factor associated
with $F$, expressed in elliptic coordinates. We have
\begin{equation*}
  h(\xi,\eta)=\left|\cosh'(\zeta)\right|=\left|\sinh(\zeta)\right|
  =\left|\sinh(\xi)\cos(\eta)+i\cosh(\xi)\sin(\eta)\right|=\sqrt{\cosh^2(\xi)-\cos^2(\eta)}.
\end{equation*}
For any function $u\in\mathcal
  Q$, let us define $U:=u\circ F$. From the fact the $F$ is conformal, it follows that $\left|\nabla U\right|$ is in $L^2(D)$ with
\begin{equation*}
	\int_{D}\left|\nabla U\right|^2\,d\xi d\eta=\int_{\RR^2_+}\left|\nabla u\right|^2\,dx.
\end{equation*} 
We also have
\begin{equation}\label{eq:normder}
	\frac{\partial u}{\partial \nu}(x)=-\frac1{h(0,\eta)}\frac{\partial  U}{\partial \xi}(0,\eta)
\end{equation}
for any $x\in \Gamma_1$, where $\eta\in(0,\pi)$ satisfies $x=F(0,\eta)=(\cos(\eta),0)$. Furthermore, $U$ is harmonic in $D$ if, and only if, $u$ is harmonic in $\RR^2_+$. 

We now give an explicit formula for $w_k\circ F$.
\begin{prop} \label{propwk} For any positive integer $k$, $w_k\circ
  F=W_k$, where $W_k$ is defined in \eqref{eq:W_k}.
\end{prop}

\begin{proof} Let us begin by computing the function $\Psi_k:=\psi_k\circ F$. We have $\psi_k(x)=\mbox{Im}\left(z^k\right)$, so that $\Psi_k(\xi,\eta)=\mbox{Im}\left((\cosh(\zeta))^k\right)$, where the complex variables $z$ and $\zeta$ are defined as above. Using the binomial theorem, we find
\begin{equation*}
  \Psi_k(\xi,\eta)=\mbox{Im}\left(\frac1{2^k}
    \sum_{j=0}^k\left(\begin{array}{c}k\\
                        j\end{array}\right)e^{(k-2j)\zeta}\right)=
                  \frac1{2^k}\sum_{j=0}^k\left(\begin{array}{c}k\\
                                                 j\end{array}\right)
                                             e^{(k-2j)\xi}\sin\left((k-2j)\eta\right).
\end{equation*}
This can be written
\begin{equation*}
	\Psi_k(\xi,\eta)=\frac1{2^{k-1}}\sum_{j=0}^{\left\lfloor\frac{k-1}2\right\rfloor}\left(\begin{array}{c}k\\ j\end{array}\right)\sinh\left((k-2j)\xi\right)\sin\left((k-2j)\eta\right)
\end{equation*}
by grouping terms of the sum in pairs, starting from opposite extremities. In particular, for all $\eta\in(0,\pi)$,
\begin{equation*}
	\frac{\partial \Psi_k}{\partial \xi}(0,\eta)=\frac1{2^{k-1}}\sum_{j=0}^{\left\lfloor\frac{k-1}2\right\rfloor}(k-2j)\left(\begin{array}{c}k\\ j\end{array}\right)\sin\left((k-2j)\eta\right).
\end{equation*}
We now define 
\begin{equation*}
	V(\xi,\eta)=\frac1{2^{k-1}}\sum_{j=0}^{\left\lfloor\frac{k-1}2\right\rfloor}\left(\begin{array}{c}k\\ j\end{array}\right)e^{-(k-2j)\xi}\sin\left((k-2j)\eta\right).
\end{equation*}
The function $|\nabla V|$ is in $L^2(D)$ and, for all $\eta\in(0,\pi)$,
\begin{equation*}
	\frac{\partial V}{\partial \xi}(0,\eta)=
-\frac1{2^{k-1}}\sum_{j=0}^{\left\lfloor\frac{k-1}2\right\rfloor}
(k-2j)\left(\begin{array}{c}k\\ j\end{array}\right)\sin\left((k-2j)\eta\right).
\end{equation*}
Additionally, $V$ vanishes on half-lines defined by $\eta=0$ and
$\eta=\pi$, which are the lower and upper boundary of $D$, respectively,
and are mapped to $\RR\times\{0\}\setminus\Gamma_1$ by
$F$. It can be checked directly that $V\circ F^{-1}\in \mathcal Q$. Finally, $V$ is harmonic in $D$, since it is a
linear combination of functions of the type
$(\xi,\eta)\mapsto e^{\pm n \xi}e^{\pm i n \eta}$, which are harmonic. We conclude that $V\circ F^{-1}$ is a solution of
Problem \eqref{eq:wk}, and therefore $V=w_k\circ F$ by uniqueness.
\end{proof}

\begin{proof}[Proof of Theorem \ref{t:blowup2}]
   Theorem \ref{t:blowup2} follows combining Theorem
   \ref{t:gad} and Proposition \ref{propwk}. 
\end{proof}

\begin{cor} \label{cormk} For any positive integer $k\ge1$,
\begin{equation}\label{eq:const1}
	\int_{-1}^1 \frac{\partial w_k}{\partial x_2}\,w_k \, dx_2
	=-\frac\pi{2^{2k-1}}\sum_{j=0}^{\left\lfloor\frac{k-1}2\right\rfloor}(k-2j)\left(\begin{array}{c}k\\ j\end{array}\right)^2.
\end{equation}
\end{cor}

\begin{proof}
Using \eqref{eq:W_k}, a direct computation gives
\begin{equation*}
	\nabla W_k(\xi,\eta)=\frac1{2^{k-1}}\sum_{j=0}^{\left\lfloor\frac{k-1}2\right\rfloor}(k-2j)\left(\begin{array}{c}k\\ j\end{array}\right)e^{-(k-2j)\xi}(-\sin\left((k-2j)\eta\right),\cos\left((k-2j)\eta\right).
\end{equation*}
Recalling \eqref{eq:normder}, we perform a standard change of variables in the left-hand side of \eqref{eq:const1}
to elliptic coordinates and this yields the thesis.
\end{proof}

\subsection{Simplification of the sum}

We now prove  the following result.

\begin{lem} \label{lemSumBin} For every integer $k\ge1$,
\begin{equation*}
	\sum_{j=0}^{\left\lfloor\frac{k-1}2\right\rfloor}(k-2j)\left(\begin{array}{c} k \\ j\end{array}\right)^{\!\!2}=k\left(\begin{array}{c}k-1\\ \left\lfloor\frac{k-1}2\right\rfloor\end{array}\right)^{\!\!2}.
\end{equation*}
\end{lem}
\begin{proof}We will use repeatedly the two following properties of binomial coefficients. First, the \emph{Vandermonde identity}: for any non-negative integers $m$, $n$ and $r$,
\begin{equation}
\label{eqVdM}
	\sum_{j=0}^r\left(\begin{array}{c} m \\ j\end{array}\right)\left(\begin{array}{c} n \\ r-j\end{array}\right)=\left(\begin{array}{c} m+n \\ r\end{array}\right);
\end{equation}
and second, the elementary identity
\begin{equation}
\label{eqElem}
	\left(\begin{array}{c} n \\ r\end{array}\right)=\frac{n}{r}\left(\begin{array}{c} n-1 \\ r-1\end{array}\right)
\end{equation}
with $n$ and $r$ positive integers.

Let us now fix an integer $k\ge 1$. To simplify the notation, we write 
\begin{equation*}
s:=\left\lfloor\frac{k-1}2\right\rfloor \hspace{1cm} \mbox{ and } \hspace{1cm}  S:=\sum_{j=0}^s(k-2j)\left(\begin{array}{c} k \\ j\end{array}\right)^{\!\!2}.
\end{equation*}
Next, we remark that 
\begin{equation*}
	S=S_0-\frac2kS_1-\frac2kS_2,
\end{equation*}
with
\begin{equation*}
	S_0:=\sum_{j=0}^s k\left(\begin{array}{c} k \\
                                   j\end{array}\right)^{\!\!2},\quad 
S_1:=\sum_{j=0}^s j(k-j)\left(\begin{array}{c} k \\ j\end{array}\right)^{\!\!2},\quad
	S_2:=\sum_{j=0}^s j^2\left(\begin{array}{c} k \\ j\end{array}\right)^{\!\!2}.
      \end{equation*}
Let us compute the previous sums when $k=2p+1$, with $p$ a non-negative integer. We first have
\begin{equation*}
	\frac{S_0}k=\frac12\sum_{j=0}^k\left(\begin{array}{c} k \\ j\end{array}\right)^{\!\!2}=\frac12\left(\begin{array}{c} 2k \\ k\end{array}\right),
\end{equation*}
where the last equality is a special case of identity \eqref{eqVdM}. We then find
\begin{multline*}
  S_1=\frac12\sum_{j=0}^kj\left(\begin{array}{c} k \\
      j\end{array}\right)(k-j)\left(\begin{array}{c} k \\
      k-j\end{array}\right)=\frac{k^2}2\sum_{j=1}^{k-1}\left(\begin{array}{c}
      k-1 \\ j-1\end{array}\right)\left(\begin{array}{c} k-1 \\
      k-j-1\end{array}\right)\\
 = \frac{k^2}2\sum_{\ell=0}^{k-2}\left(\begin{array}{c} k-1 \\
      \ell\end{array}\right)\left(\begin{array}{c} k-1 \\
      k-2-\ell\end{array}\right)=\frac{k^2}2\left(\begin{array}{c}
      2k-2 \\ k-2\end{array}\right)
\end{multline*}
by applying Identity \eqref{eqElem} followed by \eqref{eqVdM}. Finally, Identity \eqref{eqElem} implies
\begin{multline*}
  S_2=k^2\sum_{j=1}^p\left(\begin{array}{c} k-1 \\
      j-1\end{array}\right)^{\!\!2}=k^2\sum_{\ell=0}^{p-1}\left(\begin{array}{c}
      k-1 \\ \ell\end{array}\right)^{\!\!2} \\
 = \frac{k^2}2\left(\sum_{\ell=0}^{k-1}\left(\begin{array}{c} k-1 \\
        \ell\end{array}\right)^{\!\!2}-\left(\begin{array}{c} k-1 \\
        p\end{array}\right)^{\!\!2}\right)=\frac{k^2}2\left(\begin{array}{c}
      2k-2 \\ k-1\end{array}\right)-\frac{k^2}2\left(\begin{array}{c}
      k-1 \\ p\end{array}\right)^{\!\!2}.
\end{multline*}
We obtain
\begin{multline*}
	S=\frac{k}2
\left(\begin{array}{c} 2k \\ k\end{array}\right)-k\left(
\begin{array}{c} 2k-2 \\ k-2\end{array}\right)-k\left(\begin{array}{c} 2k-2 \\
	k-1\end{array}\right)+k\left(\begin{array}{c} k-1 \\ p\end{array}\right)^{\!\!2} \\
=\frac{k}2\left(\begin{array}{c} 2k \\ k\end{array}\right)-k
\left(\begin{array}{c} 2k-1 \\
        k-1\end{array}\right)+k\left(\begin{array}{c} k-1 \\
                                       p\end{array}\right)^{\!\!2}
=k\left(\begin{array}{c} k-1 \\ p\end{array}\right)^{\!\!2}
\end{multline*}
where the second equality follows from Pascal's identity and the third from Identity \eqref{eqElem}.

Let us now treat the case $k=2p$, with $p$ a positive integer. In a similar way as before, we find
\begin{equation*}
	S_0=\frac{k}2\left(\sum_{j=0}^k\left(\begin{array}{c} k \\ j\end{array}\right)^{\!\!2}-\left(\begin{array}{c} k \\ p\end{array}\right)^{\!\!2}\right)=\frac{k}2\left(\begin{array}{c} 2k \\ k\end{array}\right)-\frac{k}2\left(\begin{array}{c} k \\ p\end{array}\right)^{\!\!2},
\end{equation*}
\begin{equation*}	
	S_1=\frac12\left(\sum_{j=0}^kj\left(\begin{array}{c} k \\ j\end{array}\right)(k-j)\left(\begin{array}{c} k \\ k-j\end{array}\right)-p^2\left(\begin{array}{c} k \\ p\end{array}\right)^{\!\!2}\right)=\frac{k^2}2\left(\begin{array}{c} 2k-2 \\ k-2\end{array}\right)-\frac{k^2}8\left(\begin{array}{c} k \\ p\end{array}\right)^{\!\!2}
	\end{equation*}
and
\[
S_2=k^2\sum_{j=0}^{p-2}\left(\begin{array}{c} k-1 \\
                               j\end{array}\right)^{\!\!2}
  =\frac{k^2}2\left(\sum_{j=0}^{k-1}\left(\begin{array}{c} k-1
                                            \\
                                            j\end{array}\right)^{\!\!2}-2
  \left(\begin{array}{c} k-1 \\ p-1\end{array}\right)^{\!\!2}\right)
  =\frac{k^2}2\left(\begin{array}{c}
                      2k-2 \\
                      k-1\end{array}\right)-k^2\left(\begin{array}{c}
                                                       k-1 \\
                                                       p-1\end{array}\right)^{\!\!2}.
\]
We finally obtain, after simplifications,
\begin{equation*}
	S=k\left(\begin{array}{c} k-1 \\ p-1\end{array}\right)^2.
\end{equation*}
This completes the proof of Lemma \ref{lemSumBin}. 
\end{proof}

\subsection{Conclusions}

By the results from the preceding subsections, we can now prove Proposition \ref{propmk} and Theorem \ref{t:new_main}.

\begin{proof}[Proof of Proposition \ref{propmk}]
It follows from Corollary \ref{cormk} and Lemma \ref{lemSumBin}.
\end{proof}

Combining the above results, we can now prove our main theorem.
\begin{proof}[Proof of Theorem \ref{t:new_main}]
   Theorem \ref{t:new_main} follows from the combination of Theorem
   \ref{t:gad} and Proposition \ref{propmk}. 
\end{proof}

\section{Asymptotic estimates for Aharonov--Bohm eigenvalues} \label{s:AB}

\subsection{Symmetry for the Aharonov--Bohm operator}

As in Section \ref{subs:IntroAB}, we assume
$\widehat\Omega\subset \RR^2$ to be a bounded open set with a
Lipschitz boundary, such that $0\in
  \widehat\Omega$.
We additionally assume that $\widehat\Omega$ is symmetric with respect
to the $x_1$-axis and that $\Omega:=\widehat\Omega\cap\RR^2_+$ also
has a Lipschitz boundary.

\begin{figure}
\centering
 \begin{tikzpicture}[scale=0.5]
 \draw[rounded corners=1.5ex]
 (0,0)  --
(0,3)  --
(3,6)  --
(5,4)  --
(8,3)  --
(9,2)  --
(9.3,0)  --
(9,-2)  --
(8,-3)  --
(5,-4)  --
(3,-6)  --
 (0,-3) --cycle; 
 \draw[->](4,-7) -- (4,7);
 \draw[->](-1,0) -- (10,0);
 \draw  (2,0) circle (6pt);
 \draw[->,line width=1pt](2,0) -- (3,0);
 \draw (2,-0.8) node {$-\eps$};
 \draw  (6,0) circle (6pt);
 \draw[->,line width=1pt](6,0) -- (5,0);
\draw (6,-0.8) node {$\eps$};
\end{tikzpicture}
\caption{The domain considered for Aharonov--Bohm eigenvalues with collapsing symmetric poles.}
\label{fig:2}\end{figure}
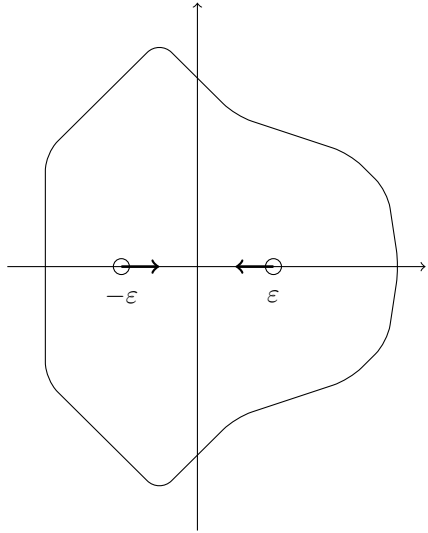

According to \cite[Theorem VIII.15]{ReeSim80}, there exists a unique
Friedrichs extension $H_\eps$ of the quadratic form $q^{AB}_\eps$, that is
to say a self-adjoint operator whose domain $\mathcal D(H_\eps)$ is
contained in $\mathcal Q^{AB}_\eps$ and which satisfies
\begin{equation*}
  \langle H_\eps u,v \rangle= q^{AB}_\eps(u,v)
=\int_{\widehat\Omega}(i\nabla+\vA_\eps)u\cdot
\overline{(i\nabla+\vA_\eps)v}\,dx\quad
\mbox{for all } u,v\in \mathcal D(H_\eps),
\end{equation*}
where we are denoting by $q^{AB}_\eps$ both the quadratic
 form defined in \eqref{eq:QuadAB2Poles} and the associated bilinear form (see Figure \ref{fig:2}).
We recall in this section the results proved in \cite{AFHL2016}
concerning the  properties of $H_\eps$, in particular the effect of
the symmetry of the domain on its spectrum. Since most of the
  proofs in the present section reduce to a series of standard
  verifications, we generally only give an indication of them. We use
gauge functions $\Phi_\eps$, for $\eps \in (0,\eps_0]$, whose
existence is guaranteed by the following result. In
  the sequel the denote as $\sigma$ the reflection through the
  $x_1$-axis, i.e. $\sigma(x_1,x_2)=(x_1,-x_2)$.
\begin{lem} \label{l:gauge} For each $\eps>0$, there exists a function $\Phi_{\eps}$ in $C^{\infty}\left(\RR^2\setminus \Gamma_\eps\right)$ satisfying
\begin{enumerate}[(i)]
	\item $\Phi_\eps\circ\sigma=\overline{\Phi}_\eps$ in $\RR^2\setminus \Gamma_\eps$; 
	\item $\left|\Phi_\eps\right|=1$ in $\RR^2\setminus \Gamma_\eps$;
	 \item $\left(i\nabla+\vA_\eps\right)\Phi_\eps=0$ in $\RR^2\setminus \Gamma_\eps$;
	 \item $\Phi_\eps=1$ on $(\RR\times\{0\})\setminus\Gamma_\eps$ and $\lim_{\delta\to0^+}\Phi_\eps(t,\pm \delta)=\pm i$ for every $t\in(-\eps,\eps)$.
\end{enumerate}
\end{lem} 

We define the anti-unitary operators $K_\eps$ and $\Sigma^c$ by
$K_\eps u:=\Phi_\eps^2 \overline u$ and
$\Sigma^c u:=\overline u\circ \sigma$. The subspace
$\mathcal D(H_\eps)\subset \mathcal H$ is preserved by $K_\eps$ and
$\Sigma^c$. The operators $K_\eps$, $\Sigma^c$ and $H_\eps$ mutually
commute. In particular, we can define the following subsets
\begin{align*}
	\mathcal H_{K,\eps}&:=\{u\in\mathcal H:\,K_\eps u=u\};\\
	\mathcal D(H_{K,\eps})&:=\{u\in\mathcal D(H_\eps):\,K_\eps u=u\}.
\end{align*} 
The scalar product $\langle\cdot\,,\cdot\rangle$ gives $\mathcal H_{K,\eps}$ the structure of a real Hilbert space. As suggested by the notation, we define $H_{K,\eps}$ as the restriction of $H_\eps$ to $\mathcal D(H_{K,\eps})$. It is a positive self-adjoint operator on $\mathcal H_{K,\eps}$ of domain $\mathcal D(H_{K,\eps})$, with compact resolvent. It has the same eigenvalues as $H_\eps$, with the same multiplicities.  The fact that $K$ and $\Sigma^c$ commute ensures that $\mathcal H_{K,\eps}$ and $\mathcal D(H_{K,\eps})$ are $\Sigma^c$-invariant. We can therefore define
\begin{align*}
	\mathcal H_{K,\eps}^s&:=\{u\in\mathcal H_{K,\eps}:\Sigma^cu=u\};\\
	\mathcal D(H_{K,\eps}^s)&:=\{u\in\mathcal D(H_{K,\eps}):\Sigma^cu=u\};\\
	\mathcal H_{K,\eps}^a&:=\{u\in\mathcal H_{K,\eps}:\Sigma^cu=-u\};\\
	\mathcal D(H_{K,\eps}^a)&:=\{u\in\mathcal D(H_{K,\eps}):\Sigma^cu=-u\}.
\end{align*} 
We have the following orthogonal decomposition of $\mathcal H_{K,\eps}$ into spaces of symmetric and antisymmetric functions:
\begin{equation}
\label{eq:MagnDecomp}
	\mathcal H_{K,\eps}=\mathcal H_{K,\eps}^s \oplus \mathcal H_{K,\eps}^a.
\end{equation}
We also define $H_{K,\eps}^s$ and $H_{K,\eps}^a$ as the restrictions
of $H_{K,\eps}$ to $\mathcal D(H_{K,\eps}^s)$ and $\mathcal
D(H_{K,\eps}^a)$ respectively. The operator $H_{K,\eps}^s$ is positive
and self-adjoint on $\mathcal H_{K,\eps}^s$ of domain $\mathcal
D(H_{K,\eps}^s)$ , with compact resolvent. Similar
  conclusions hold for $\mathcal H_{K,\eps}^a$. Decomposition \eqref{eq:MagnDecomp} implies the following result.
\begin{lem}	The spectrum of $H_{K,\eps}$ is the union of the spectra of $H_{K,\eps}^s$ and $H_{K,\eps}^a$, counted with multiplicities.
\end{lem}
\begin{rem} Let us note that we can give an alternative description of the spectra of $H_{K,\eps}^s$ and $H_{K,\eps}^a$. One can check that they are the spectra of the quadratic form $q^{AB}_\eps$ restricted to $\mathcal Q^{AB}_\eps \cap \mathcal H_{K,\eps}^s$ and $\mathcal Q^{AB}_\eps \cap \mathcal H_{K,\eps}^a$ respectively. These spectra can therefore be obtained by the min-max principle.
\end{rem}

\subsection{Isospectrality}

In this subsection, we establish an isospectrality result between  Aharonov-Bohm 
eigenvalue problems with symmetry and Laplacian eigenvalue problems
with mixed boundary conditions, in the spirit of \cite{BN-He-HH2009}.

To this aim, we define an additional family of eigenvalue problems, similar to Problems \eqref{eqD} and \eqref{eqDND}. With the notation $\partial\Omega_+:=\partial \Omega\cap \RR_+^2$ and $\partial\Omega_0:=\partial \Omega\cap (\RR\times\{0\})$, we consider the eigenvalue problem 
\begin{equation}\label{eqN}
\begin{cases}
-\Delta\,u=\lambda\,u,& \mbox{in }\Omega,\\
u=0,&\mbox{on } \partial\Omega_+,\\
\frac{\partial u}{\partial \nu}=0,&\mbox{on } \partial\Omega_0.
\end{cases}
\end{equation}
We denote by $(\mu_j)_{j\ge1}$ the eigenvalues of Problem \eqref{eqN}. We also consider, for each $\eps\in(0,\eps_0]$,
\begin{equation}\label{eqNDN}
\begin{cases}
-\Delta\,u=\lambda\,u,& \mbox{in }\Omega,\\
u=0,&\mbox{on } \partial\Omega_+ \cup\Gamma_{\eps},\\
\frac{\partial u}{\partial \nu}=0,&\mbox{on } \partial\Omega_0\setminus\Gamma_{\eps},
\end{cases}
\end{equation}
and denote by  $(\mu_j(\eps))_{j\ge1}$ the corresponding eigenvalues. In order to give a rigorous definitions, we use a weak formulation.  We define
\begin{equation*}
  \mathcal{R}_0=\left\{u\in H^1(\Omega)\,;\,\chi_{\partial\Omega_+}\gamma_0u=0\mbox{ in } L^2(\partial \Omega)\right\},
\end{equation*}
and, for  $\eps\in(0,\eps_0]$,
\begin{equation*}
  \mathcal{R}_\eps=\left\{u\in H^1(\Omega)\,;\,\chi_{\partial\Omega_+\cup\Gamma_{\eps}}\gamma_0u=0\mbox{ in } L^2(\partial \Omega)\right\}.
\end{equation*}
We denote by $r_0$ and $r_\eps$ the restriction of the quadratic form $q$, defined in Equation \eqref{eqQuad}, to $\mathcal{R}_0$ and $\mathcal{R}_\eps$ respectively. We then define $(\mu_j)_{j\ge1}$ and $(\mu_j(\eps))_{j\ge1}$ as,  respectively, the eigenvalues of the quadratic forms $r_0$ and $r_\eps$; they are obtained by the min-max principle. 

\begin{rem}  \label{remDecompLap} We can give another interpretation of the eigenvalues $(\mu_j)_{j\ge1}$ and $(\lambda_j)_{j\ge1}$. Using the unitary operator $\Sigma: u\mapsto u\circ\sigma$, we obtain a orthogonal decomposition of $L^2\big(\widehat\Omega\big)$ into symmetric and antisymmetric functions:
\begin{equation}
	\label{eq:DecompLap}
	L^2\big(\widehat\Omega\big)=\mbox{ker}\left(I-\Sigma\right)\oplus\mbox{ker}\left(I+\Sigma\right).
\end{equation}
This decomposition is preserved by the action of the Dirichlet Laplacian $-\widehat \Delta$, and we can therefore define $-\Delta^s$ (resp. $-\Delta^a$) as the restriction of $-\widehat \Delta$ to symmetric (resp. antisymmetric) functions in the domain of $-\widehat \Delta$. One can then check that $(\mu_j)_{j\ge1}$ is the spectrum of  $-\Delta^s$ and $(\lambda_j)_{j\ge1}$ is the spectrum of $-\Delta^a$.
\end{rem}

It remains to connect the eigenvalues of Problems \eqref{eqNDN} and \eqref{eqDND} to the eigenvalues of $H_\eps$. To this end, we define the following linear operator, which performs a gauge transformation:
\begin{equation*}
	\begin{array}{cccc}
		U_\eps:&\mathcal H&\to&L^2(\Omega,\CC)\\[3pt]
		&u&\mapsto&\sqrt2\,\overline\Phi_\eps u_{|\Omega}.
	\end{array}
\end{equation*}
We 
recall that $L^2(\Omega)$ denotes the real Hilbert space of real-valued $L^2$ functions in $\Omega$. We have the following result.
\begin{lem}\label{l:IsoAS}
The operator $U_\eps$ satisfies the following properties:
\begin{enumerate}[(i)]
	\item $U_\eps\left(\mathcal H_{K,\eps}\right)\subset L^2(\Omega)$ and $U_\eps\left(\mathcal Q^{AB}_\eps\right)\subset H^1(\Omega,\CC)$;
	\item $U_\eps$ induces a real-unitary bijective map from
          $\mathcal Q^{AB}_\eps \cap \mathcal H_{K,\eps}^s$ to
          $\mathcal R_\eps$ such that $q^{AB}_\eps (u)=q\left(U_\eps
            u\right)$ for all $u\in \mathcal Q^{AB}_\eps \cap \mathcal H_{K,\eps}^s$;
	\item $U_\eps$ induces a real-unitary bijective map from
          $\mathcal Q^{AB}_\eps \cap \mathcal H_{K,\eps}^a$ to
          $\mathcal Q_\eps$ such that $q^{AB}_\eps (u)=q\left(U_\eps
            u\right)$ for all $u\in \mathcal Q^{AB}_\eps \cap \mathcal H_{K,\eps}^a$.
\end{enumerate}
\end{lem}

\begin{proof} If $u\in \mathcal H_{K,\eps}$, then
  $u=\Phi_\eps^2\overline u$, so that
  $\overline \Phi_\eps u= \overline{\overline \Phi_\eps u}$, that is
  to say $\overline \Phi_\eps u$ is real-valued. This proves the first
  half of (i). For the second half, let us assume that
  $u\in\mathcal Q^{AB}_\eps$. Using the definition of
  $\mathcal Q^{AB}_\eps$, given in Equation \eqref{eq:QAB2Poles}, and
  Property (iii) of Lemma \ref{l:gauge}, we find the following
  identity, in the sense of distributions in $\Omega$:
\begin{equation*}
  \nabla\left(\overline \Phi_\eps u\right)=\overline
    \Phi_\eps\nabla
  u+\nabla\left(\overline\Phi_\eps\right)u=\overline
    \Phi_\eps\left(\nabla-iA_\eps\right)u\quad\text{in
  $\Omega$}.
\end{equation*}
This proves that $\overline\Phi_\eps u_{|\Omega} \in 
H^1(\Omega,\CC)$ and that 
\begin{equation*}
	\int_{\Omega}\left|\left(\nabla-iA_\eps\right)u\right|^2\,dx=
\int_{\Omega}\left|\nabla\left(\overline\Phi_\eps u\right)\right|^2\,dx.
\end{equation*}
Let us now additionally assume that $u\in \mathcal Q^{AB}_\eps \cap
\mathcal H_{K,\eps}^s$. Since $\Sigma^cu=u$, Property (i) of Lemma
\ref{l:gauge} implies that $(\overline\Phi_\eps
  u)\circ\sigma=\Phi_\eps
  \overline u$. Therefore,
\begin{equation*}
	\int_{\widehat\Omega}\left|u\right|^2\,dx=2\int_{\Omega}\left|\overline\Phi_\eps u\right|^2\,dx=\int_{\Omega}\left|U_\eps u\right|^2\,dx.
\end{equation*}
Furthermore, Property (iv) of Lemma \ref{l:gauge} and the equation
$\Sigma^cu=u$ imply that $u$ vanishes on
$\Gamma_\eps$, hence $U_\eps u\in \mathcal R_\eps$.  This implies that $\overline\Phi_\eps
  u\in H^1(\widehat \Omega)$
and 
\begin{equation*}
	\int_{\widehat\Omega}\left|\left(\nabla-iA_\eps\right)u\right|^2\,dx
=\int_{\widehat\Omega}\left|\nabla\left(\overline\Phi_\eps u\right)\right|^2\,dx=
2\int_{\Omega}\left|\nabla\left(\overline\Phi_\eps u\right)\right|^2\,dx=\int_{\Omega}\left|\nabla\left(U_\eps u\right)\right|^2\,dx.
\end{equation*}
We conclude  that the mapping $U_\eps: \mathcal Q^{AB}_\eps \cap \mathcal H_{K,\eps}^s\to\mathcal R_\eps$ is well-defined, real-unitary, and that $q^{AB}_\eps (u)=q\left(U_\eps u\right)$. To show that the mapping is bijective, we consider the operator $V_\eps$ defined in the following way: given $v\in L^2(\Omega)$, we denote by $\widetilde v$ its extension by symmetry to $\widehat \Omega$ and we set
\begin{equation*}
	V_\eps v:=\frac1{\sqrt2}\Phi_\eps \widetilde v.
\end{equation*}
It can be checked, in a way similar to what has been done for $U_\eps$, that $V_\eps$ induces the inverse of $U_\eps$, from $\mathcal R_\eps$ to $Q^{AB}_\eps \cap \mathcal H_{K,\eps}^s$. This proves (ii). The proof of (iii) is similar, the difference being that we must check that $\overline\Phi_\eps u$ vanishes on $(\RR\times\{0\})\setminus\Gamma_\eps$ when $u\in Q^{AB}_\eps \cap \mathcal H_{K,\eps}^a$.
\end{proof}

\begin{cor} \label{c:MagnEV} The spectra of $H_{K,\eps}^s$ and $H_{K,\eps}^a$ are $(\mu_j(\eps))_{j\ge1}$ and $(\lambda_j(\eps))_{j\ge1}$ respectively.
\end{cor}

\subsection{Eigenvalues variations}
\label{subs:EigVar}

Let us first state some auxiliary results, which we prove in  Appendix \ref{a:B}. 

\begin{prop} \label{p:ConvEven} For all $N\in\NN^*$, $\mu_N(\eps)\to \mu_N$ as $\eps\to 0$.
\end{prop}

\begin{prop} \label{p:AsymptNDN} Let $\mu_N$ be a simple eigenvalue of
   $-\Delta^s$ (see Remark \ref{remDecompLap}) and $u_N$ be an
  associated eigenfunction, normalized in
  $L^2\big(\widehat\Omega\big)$. If $u_N(0)\neq0$, then
\begin{equation*}
	\mu_N(\eps)=\mu_N+\frac{2\pi}{|\log(\eps)|}u_N^2
        (0)+o\left(\frac1{|\log(\eps)|}\right)
\quad\text{as }\eps\to0.
\end{equation*}
If  
\begin{equation*}
	r^{-k} u_N(r\cos t,r\sin t)\to\widehat\beta \cos\left(k
          t\right) 
\mbox{ in } C^{1,\tau}\left([0,\pi],\RR\right)
\end{equation*}
as $r\to 0^+$ for all $\tau\in (0,1)$, with $k\in \NN^*$ and $\widehat\beta\in \RR\setminus \{0\}$, then
\begin{equation*}
	\mu_N(\eps)=\mu_N+\frac{k\pi{\widehat\beta}^2}{4^{k-1}}\left(\begin{array}{c}
                                                                           k-1\\
                                                                           \left\lfloor\frac{k-1}2\right\rfloor\end{array}\right)^2\eps^{2k}+o\left(\eps^{2k}\right)
\quad\text{as }\eps\to0.
\end{equation*}
\end{prop}

We now prove Theorem \ref{t:dipoleSymOdd}. Since
$\widehat u_N$ is odd in $x_2$, $\widehat \lambda_N$
belongs to the spectrum of $-\Delta^a$.  Since
$\widehat\lambda_N$ is simple, it does not belong to the spectrum of
$-\Delta^s$, according to the orthogonal decomposition
\eqref{eq:DecompLap}. It follows from
Remark \ref{remDecompLap} that there exists $K\in \NN^*$ such that $\widehat \lambda_N=\lambda_K$ 
and that $\lambda_K$ is a simple eigenvalue of $q_0$ 
  in $\Omega$. By continuity, $\lambda_K(\eps)\to \lambda_K$ as $\eps\to 0^+$.

From Corollary \ref{c:MagnEV}, Proposition \ref{p:ConvEven} and the
fact that $\widehat\lambda_N$ is simple, it follows that there exists
$\eps_1>0$ such that $\lambda^{AB}_N(\eps)=\lambda_K(\eps)$ for every
$\eps\in (0,\eps_1)$.  The conclusion of Theorem \ref{t:dipoleSymOdd}
follows from Theorem \ref{t:new_main}, using the fact that $\lambda_K$
is simple.
Let us note that the eigenfunction $\widehat u_N$ in Theorem
   \ref{t:dipoleSymOdd} is normalized in
   $L^2\big(\widehat\Omega\big)$, while the eigenfunction $u_N$ in
   Theorem \ref{t:new_main} is normalized in $L^2\big(\Omega\big)$. We
   therefore have to apply Theorem \ref{t:new_main} with
 $\beta=\sqrt 2\,\widehat\beta$ to obtain the correct result.

We can use the results of the preceding sections to study some
multiple eigenvalues. Let $\widehat\lambda_N$ be an eigenvalue of
$-\Delta$ on $\widehat\Omega$, possibly multiple. We define 
\begin{equation*}
	N_0:=\min\left\{M\in \NN^*\,;\,\widehat\lambda_M=\widehat\lambda_N\right\} \mbox{ and }
	N_1:=\max\left\{M\in \NN^*\,;\,\widehat\lambda_M=\widehat\lambda_N\right\}.
\end{equation*}
According to 
Remark \ref{remDecompLap}, there exists $K\in \NN^*$ such that $\widehat\lambda_N=\lambda_K$ or there exists $L\in \NN^*$ such that $\widehat\lambda_N=\mu_L$.

\begin{prop} \label{p:MultOdd} Let us assume that $\widehat\lambda_N=\lambda_K$ with $K\in \NN^*$ and that $\lambda_K$ is a simple eigenvalue of $q_0$. Let us denote by $u_K$ an associated normalized eigenfunction for $q_0$, and let us assume that 
\begin{equation*}
	r^{-k} u_K(r\cos t,r\sin t)\to \beta \sin\left(k t\right) \mbox{ in } C^{1,\tau}\left([0,\pi],\RR\right)
\end{equation*}
as $r\to 0^+$ for all $\tau\in (0,1)$, with $k\in \NN^*$ and $\beta\in \RR\setminus \{0\}$. Then
\begin{equation*}
	\lambda_{N_0}^{AB}(\eps)=\widehat\lambda_N-\frac{k\pi\beta^2}{2^{2k-1}}\left(\begin{array}{c}
                                                                                       k-1\\
                                                                                       \left\lfloor\frac{k-1}2\right\rfloor\end{array}\right)^2\eps^{2k}+o\left(\eps^{2k}\right)\quad\text{as
                                                                                   }\eps\to0.
\end{equation*}
\end{prop}

\begin{proof} Let us set $m:=N_1-N_0+1$, the multiplicity of $\widehat\lambda_N$. If $m=1$, the conclusion follows from Theorem \ref{t:dipoleSymOdd}. We therefore assume $m\ge2$ in the rest of the proof. 
Remark \ref{remDecompLap} and the fact that $\lambda_K$ is simple imply that there exists $L\in \NN^*$ such that $\mu_L=\mu_{L+1}=\dots=\mu_{L+m-2}=\widehat\lambda_N$. 
From Proposition \ref{p:ConvEven}, we deduce that there exists $\eps_1>0$ such that, for every $\eps\in (0,\eps_1)$, 
\begin{equation*}
	\left\{\lambda^{AB}_{N_0}(\eps);\lambda^{AB}_{N_0+1}(\eps),\dots,\lambda^{AB}_{N_1}(\eps)\right\}=\left\{\lambda_K(\eps),\mu_L(\eps),\dots,\mu_{L+m-2}(\eps)\right\}.
\end{equation*} 
The function $\eps\mapsto \lambda_K(\eps)$ is non-increasing, and the function $\eps\mapsto\mu_j(\eps)$ is non-decreasing for every $j\in \{L,\dots,L+m-2\}$, therefore $\mu_j(\eps)\ge \mu_j=\widehat\lambda_N=\lambda_K\ge \lambda_K(\eps)$. In particular $\lambda^{AB}_{N_0}(\eps)=\lambda_K(\eps)$ for every $\eps\in (0,\eps_1)$. 
The conclusion follows from Theorem \ref{t:new_main}.
\end{proof}

\begin{prop}\label{p:MultEven} Let us assume that
  $\widehat\lambda_N=\mu_L$ with $L\in \NN^*$ and that $\mu_L$ is a
  simple eigenvalue of $-\Delta^s$. Let us denote by $u_L$ an
  associated eigenfunction for $-\Delta^s$, normalized in 
  $L^2(\widehat \Omega)$.  If $u_L(0)\neq0$, then
\begin{equation*}
  \lambda_{N_1}^{AB}(\eps)=\widehat\lambda_N+
  \frac{2\pi}{|\log(\eps)|}u_L^2 (0)+o\left(\frac1{|\log(\eps)|}\right) \quad\text{as
  }\eps\to0.
\end{equation*}
If  
\begin{equation*}
	r^{-k} u_L(r\cos t,r\sin t)\to \widehat\beta \cos\left(k t\right) \mbox{ in } C^{1,\tau}\left([0,\pi],\RR\right)
\end{equation*}
as $r\to 0^+$ for all $\tau\in (0,1)$, with $k\in \NN^*$ and $\widehat\beta\in \RR\setminus \{0\}$, then
\begin{equation*}
	\lambda_{N_1}^{AB}(\eps)=\widehat\lambda_N+\frac{k\pi{\widehat\beta}^2}{4^{k-1}}\left(\begin{array}{c} k-1\\ \left\lfloor\frac{k-1}2\right\rfloor\end{array}\right)^2\eps^{2k}+o\left(\eps^{2k}\right) \quad\text{as
                                                                                   }\eps\to0.
\end{equation*}
\end{prop}

\begin{proof} In a similar way as in the proof of Proposition \ref{p:MultOdd},
we show that there exists $\eps_1>0$ such that, for every $\eps\in(0,\eps_1)$, $\lambda^{AB}_{N_1}(\eps)=\mu_L(\eps)$. 
The conclusion then follows from Proposition \ref{p:AsymptNDN}.
\end{proof}

\subsection{Example: the square}
\label{subs:Square}

As an application of the preceding results, let us study the first
four eigenvalues of the Dirichlet Laplacian
for the  square
\begin{equation}\label{eq:quadrato}
	\widehat\Omega:=\left( -\frac\pi2,\frac\pi2 \right)^2.
\end{equation}
The open set $\widehat\Omega$ is symmetric with respect to the $x_1$-axis. 
We define $\Omega:=\widehat\Omega\cap\RR^2_+$.  We denote by
  $(\widehat\lambda_j)_{j\ge1}$ the eigenvalues of the Dirichlet
  Laplacian on the square $\widehat \Omega$
  and, for $\eps\in (0,\pi/2)$, we consider the Aharonov-Bohm eigenvalues $\left(\lambda^{AB}_j(\eps)\right)_{j\ge1}$ defined in Section \ref{subs:IntroAB}.

 It is well known that the eigenvalues of the Dirichlet Laplacian on $\widehat\Omega$ are 
\begin{equation*}
	\widehat\lambda_{m,n}:=m^2+n^2,
\end{equation*}
with $m$ and $n$ positive integers, and that an associated orthonormal
family of eigenfunctions is given by
\begin{equation*}
	u_{m,n}(x_1,x_2)=\frac2\pi\,f_m(x_1)f_n(x_2),
\end{equation*}
where 
\begin{equation*}
	f_k(x)=	\begin{cases}
				\sin(kx),&\mbox{ if $k$ is even},\\
				\cos(kx),&\mbox{ if $k$ is odd}.
			\end{cases}
\end{equation*}

\begin{prop} \label{p:SquareSimple} Let us assume that $\widehat\lambda_N$ is simple. Then $\widehat\lambda_N=\widehat\lambda_{m,m}=2m^2$ for some positive integer $m$, and $\widehat\lambda_N$ cannot be written in any other way as a sum of squares of positive integers.  Then we have, as $\eps\to0^+$,
\begin{equation*}
	\lambda^{AB}_N(\eps)=\widehat\lambda_N+\frac8{\pi|\log(\eps)|}+o\left(\frac1{|\log(\eps)|}\right)
\end{equation*}
if $m$ is odd and 
\begin{equation*}
	\lambda^{AB}_N(\eps)=\widehat\lambda_N-\frac{m^4}{2\pi}\eps^4+o\left(\eps^4\right)
\end{equation*}
if $m$ is even.
\end{prop}

\begin{proof} 
In the case where $m$ is odd, an associated eigenfunction, normalized in $L^2(\widehat\Omega)$, is
\begin{equation*}
	u_{m,m}(x_1,x_2)=\frac2\pi\cos(mx_1)\cos(mx_2).
\end{equation*}
The first asymptotic expansion then follows from Theorem \ref{thmNonZero}. 

In the case where $m$ is even, an associated eigenfunction, normalized in $L^2(\widehat\Omega)$, is
\begin{equation*}
	u_{m,m}(x_1,x_2)=\frac2\pi\sin(mx_1)\sin(mx_2).
\end{equation*}
Then $\widehat\lambda_N=\lambda_K$, where $\lambda_K$ is a simple eigenvalue of $q_0$. Furthermore,
\begin{equation*}
	r^{-2} u_{m,m}(r\cos t,r\sin t)\to \frac{m^2}\pi \sin\left(2 t\right) \mbox{ in } C^{1,\tau}\left([0,\pi],\RR\right)
\end{equation*}
as $r\to 0^+$ for all $\tau\in (0,1)$. An application of Proposition \ref{p:MultOdd}, taking care of normalizing in $L^2(\Omega)$, gives the second asymptotic expansion. 
\end{proof}

\begin{prop} \label{p:SquareDouble} Assume that $\widehat\lambda_N=\widehat\lambda_{m,n}=m^2+n^2$ with $m$ even and $n$ odd, and that $\widehat\lambda_N$ has no other representation as a sum of two squares of positive integers, up to the exchange of $m$ and $n$. Then $\widehat\lambda_N$ has multiplicity two; up to replacing $N$ with $N-1$, we can assume that $\widehat\lambda_N=\widehat\lambda_{N+1}$. Then, as $\eps\to0^+$, 
\begin{align*}
	\lambda^{AB}_N(\eps)&=\widehat\lambda_N-\frac{4m^2}\pi\eps^2+o\left(\eps^2\right);\\
	\lambda^{AB}_{N+1}(\eps)&=\widehat\lambda_N+\frac{4m^2}\pi\eps^2+o\left(\eps^2\right).
\end{align*}
\end{prop}

\begin{proof}
The associated eigenfunctions
\begin{equation*}
	u_{m,n}(x_1,x_2)=\frac2\pi\sin(mx_1)\cos(nx_2)
\end{equation*}
and
\begin{equation*}
	u_{n,m}(x_1,x_2)=\frac2\pi\cos(nx_1)\sin(mx_2)
\end{equation*}
are normalized in $L^2(\widehat\Omega)$ and respectively symmetric and antisymmetric in the variable $x_2$. It follows that $\widehat\lambda_N=\mu_L=\lambda_K$, where $\mu_L$ is a simple eigenvalue of $r_0$ and $\lambda_K$ a simple eigenvalue of $q_0$. Furthermore,
\begin{equation*}
	r^{-1} u_{m,n}(r\cos t,r\sin t)\to \frac{2m}\pi \cos\left(t\right) \mbox{ in } C^{1,\tau}\left([0,\pi],\RR\right)
\end{equation*}
and 
\begin{equation*}
	r^{-1} u_{n,m}(r\cos t,r\sin t)\to \frac{2m}\pi \sin\left(t\right) \mbox{ in } C^{1,\tau}\left([0,\pi],\RR\right)
\end{equation*}
as $r\to 0^+$ for all $\tau\in (0,1)$.  The asymptotic expansions then follow from Propositions \ref{p:MultEven} and \ref{p:MultOdd}
\end{proof}


\begin{rem} We note that if $\widehat\lambda_N$ is even, in any
  representation $\widehat\lambda_N=m^2+n^2$, $m$ and $n$ have the
  same parity. Therefore, if $n\neq m$, $\widehat\lambda_N$ cannot be a simple
  eigenvalue either of $r_0$ or of $q_0$. On the other hand, if
  $\widehat\lambda_N$ is odd, in any representation
  $\widehat\lambda_N=m^2+n^2$, $m$ and $n$ have the opposite
  parity. 
Therefore, as soon as $\widehat\lambda_N$ can be written in at least two different ways as the sum of two squares, 
  $\widehat\lambda_N$ cannot be a simple eigenvalue either of $r_0$ or
  of $q_0$. The cases described in Propositions \ref{p:SquareSimple}
  and \ref{p:SquareDouble} are thus the only ones in which we can
  apply the results of Section \ref{subs:EigVar} for the square.
\end{rem}

The first four eigenvalues of the Dirichlet Laplacian
  on the square $\widehat\Omega$ satisfy the assumptions of either Proposition \ref{p:SquareSimple}
  or Proposition \ref{p:SquareDouble}, so we can apply the previous
  results to derive the following asymptotic expansions of the 
  Aharonov-Bohm eigenvalues $\lambda^{AB}_j(\eps)$ for $j=1,2,3,4$.
\begin{cor}
\label{c:Square}
Let $\lambda^{AB}_j(\eps)$ be  the 
  Aharonov-Bohm eigenvalues defined in
  \eqref{eq:QuadAB2Poles}--\eqref{eq:MinMaxAB1Poles} with
  $\widehat\Omega$ being the square defined in \eqref{eq:quadrato}. Then
we have, as $\eps\to 0^+$,
\begin{align*}
	\lambda^{AB}_1(\eps)&=2+\frac{8}{\pi}\frac1{|\log(\eps)|}+o\left(\frac1{|\log(\eps)|}\right);\\
	\lambda^{AB}_2(\eps)&=5-\frac{16}{\pi}\eps^2+o\left(\eps^2\right);\\
	\lambda^{AB}_3(\eps)&=5+\frac{16}{\pi}\eps^2+o\left(\eps^2\right);\\
	\lambda^{AB}_4(\eps)&=8-\frac{8}{\pi}\eps^4+o\left(\eps^4\right).
\end{align*}
\end{cor}

\subsection{Example: the disk}
\label{subs:Disk}

Let $(r,t)\in [0,1]\times [0,2\pi)$ be the polar coordinates of the disk.
It is well known that the eigenvalues of the Dirichlet Laplacian on the disk are given 
by the sequences
\[
 \{{j^2_{0,k}}\}_{k\geq 1} \cup \{ {j^2_{n,k}} \}_{n,k\geq 1}, 
\]
where $j_{n,k}$ denotes the $k$-th zero of the Bessel function $J_n$
for $n\geq 0$, $k\geq 1$. 
We recall that $j_{n,k}=j_{n',k'}$ if, and only if, $n=n'$ and $k=k'$ (see \cite[Section 15.28]{Watson}).
The first set is therefore made of simple eigenvalues; their eigenfunctions are given by the Bessel functions
\begin{equation}
\label{eqEFDiskSimple}
 u_{0,k}(r \cos t, r\sin t):=\sqrt{\tfrac1\pi}\tfrac{1}{|J_0'(j_{0,k})|} J_0(j_{0,k}r) \quad \text{ for } k\geq 1.  
\end{equation}
The second set is made of double eigenvalues whose eigenfunctions are spanned by 
\begin{align}
\label{eqEFDisks}u_{n,k}^s(r \cos t, r\sin t)&:
=\sqrt{\tfrac2\pi}\tfrac{1}{|J_n'(j_{n,k})|}J_n(j_{n,k}r)\cos nt,\\
  \label{eqEFDiska}u_{n,k}^a(r \cos t, r\sin t)&:
                                                 =\sqrt{\tfrac2\pi}\tfrac{1}{|J_n'(j_{n,k})|}
J_n(j_{n,k}r)\sin nt ,
\end{align}
for $n,k\geq 1$.
We stress that these 
eigenfunctions have $L^2$-norm equal to $1$ on the disk.  
It is convenient to recall (see \cite[Chapter III]{Watson}) that for any $n\in \NN\cup\{0\}$
\begin{equation}\label{eq:bessel}
 J_n(z) = \sum_{k=0}^{+\infty} \dfrac{(-1)^k (\tfrac12 z)^{n+2k}}{k!\ \Gamma(n+k+1)}.
\end{equation}
We denote by $\big(\widehat\lambda_j\big)_{j\ge1}$ the
  eigenvalues 
  of the Dirichlet Laplacian on the disk 
\[
D_1=\{(x_1,x_2)\in \RR^2:x_1^2+x_2^2<1\}
\]
and, for $\eps\in (0,1/2)$, we 
  consider the Aharonov-Bohm eigenvalues
  $\left(\lambda^{AB}_j(\eps)\right)_{j\ge1}$
  defined in Section~\ref{subs:IntroAB}.

  \begin{prop}\label{p:ABDisk}
    If $\widehat\lambda_N$ is simple, there exists an integer $k\ge1$ such that $\widehat\lambda_N=j_{0,k}^2$. Then 
\begin{equation}\label{eqEVDiskSimple}
\lambda_N^{AB}(\eps)=j_{0,k}^2+
\frac2{|J_0'(j_{0,k})|^2}\frac1{|\log(\eps)|}+o\left(\frac1{|\log(\eps)|}\right)
\end{equation}
as $\eps\to0^+$. If $\widehat\lambda_N$ is double, there exist integers $n\ge1$ and $k\ge1$ such that $\widehat\lambda_N=j_{n,k}^2$. Up to replacing $N$ by $N-1$, we can assume that $\widehat\lambda_N=\widehat\lambda_{N+1}$. Then, as $\eps\to0^+$,
\begin{align}
	\label{eqEVDiska}\lambda_N^{AB}(\eps)&=j_{n,k}^2-
\frac{2nj_{n,k}^{2n}}{(n!)^24^{2n-1}|J_n'(j_{n,k})|^2}\left(\begin{array}{c}n-1\\ \left\lfloor\frac{n-1}2\right\rfloor\end{array}\right)^{\!2}\eps^{2n}+o\left(\eps^{2n}\right),\\
	\label{eqEVDisks}\lambda_{N+1}^{AB}(\eps)&=j_{n,k}^2+\frac{2nj_{n,k}^{2n}}{(n!)^24^{2n-1}|J_n'(j_{n,k})|^2}\left(\begin{array}{c}n-1\\ \left\lfloor\frac{n-1}2\right\rfloor\end{array}\right)^{\!2}\eps^{2n}+o\left(\eps^{2n}\right).
\end{align}

\end{prop}

\begin{proof} We first consider the case where the eigenvalue
  $\widehat\lambda_N=j_{0,k}^2$ is simple; then an associated eigenfunction, normalized in the disk, is $u_{0,k}$ defined by Equation \eqref{eqEFDiskSimple}. It follows from Equation \eqref{eq:bessel} that
\begin{equation*}
	 u_{0,k}(0)=\sqrt{\frac1\pi}\frac1{|J_0'(j_{0,k})|}>0.
\end{equation*}
 Theorem \ref{thmNonZero} gives us the asymptotic expansion \eqref{eqEVDiskSimple}.

We then consider the case where $\widehat\lambda_N$ is double, with $\widehat\lambda_N=\widehat\lambda_{N+1}=j_{n,k}^2$,  $n,k\ge1$. We note that $j_{n,k}^2$ is a simple eigenvalue of $q_0$, and that the restriction of $\sqrt2 u_{n,k}^a$ to the upper half-disk is an associated normalized eigenfunction. It follows from Equation \eqref{eq:bessel} that  
\begin{equation*}
	 r^{-n}u_{n,k}^a(r \cos t, r\sin t)\to\sqrt{\frac2\pi}\frac1{|J_n'(j_{n,k})|}\frac1{\Gamma(n+1)}\left(\frac{j_{n,k}}2\right)^{\!n}\sin nt\quad\mbox{in } C^{1,\tau}\left([0,\pi],\RR\right)
\end{equation*}
as $r\to 0^+$. The asymptotic expansion \eqref{eqEVDiska} then follows from Proposition \ref{p:MultOdd}. In a similar way, $j_{n,k}^2$ is a simple eigenvalue of $-\Delta^s$, and $u_{n,k}^s$ is an associated normalized eigenfunction. It follows from Equation \eqref{eq:bessel} that  
\begin{equation*}
	 r^{-n}u_{n,k}^s(r \cos t, r\sin t)\to\sqrt{\frac2\pi}\frac1{|J_n'(j_{n,k})|}\frac1{\Gamma(n+1)}\left(\frac{j_{n,k}}2\right)^{\!n}\cos nt\quad\mbox{in } C^{1,\tau}\left([0,\pi],\RR\right)
\end{equation*}
as $r\to 0^+$. The asymptotic expansion \eqref{eqEVDisks} then follows from the second case of Proposition \ref{p:MultEven}.
\end{proof}

Additionally, there 
exist relations between the zeros of Bessel functions (to this aim we refer to 
\cite[Chapter XV.22]{Watson}): in particular, the positive zeros of the Bessel function $J_n$
are interlaced with those of the Bessel function $J_{{n+1}}$ and by
Porter's Theorem
there is an odd number of zeros of $J_{n+2}$ between
  two consecutive zeros of $J_n$.
Then, we have,
\[
 0 < j_{0,1} < j_{1,1} < j_{2,1} < j_{0,2} < j_{1,2} <  \ldots
\]
and hence, since $j_{3,1}>j_{2,1}$, the 
first three zeros of Bessel functions are, in order,
\[
 0 < j_{0,1} < j_{1,1} < j_{2,1}.
\]
Combining this information with Proposition \ref{p:ABDisk}, we
  find for example the following asymptotic expansions for the first few
  Aharonov-Bohm eigenvalues 
$\lambda^{AB}_j(\eps)$  on the disk $D_1$ as $\eps\to0^+$:
\begin{align*}
 \lambda^{AB}_1(\eps)&={j_{0,1}^2}+ \frac{2}{|J_0'(j_{0,1})|^2}\frac1{|\log(\eps)|}+o\left(\frac1{|\log(\eps)|}\right),\\
 \lambda^{AB}_2(\eps)&={j_{1,1}^2} - \frac12 \dfrac{j_{1,1}^2}{|J'_1(j_{1,1})|^2} \, \eps^2 + o(\eps^2),  \\
 \lambda^{AB}_3(\eps)&={j_{1,1}^2} + 
 \frac12 \dfrac{j_{1,1}^2}{|J'_1(j_{1,1})|^2} \, \eps^2 + o(\eps^2),\\
 \lambda^{AB}_4(\eps)&={j_{2,1}^2} - \frac{1}{64} \dfrac{j_{2,1}^4}{|J'_2(j_{2,1})|^2}\, \eps^4 + o(\eps^4),\\
 \lambda^{AB}_5(\eps)&={j_{2,1}^2} + 
\frac{1}{64} \dfrac{j_{2,1}^4}{|J'_2(j_{2,1})|^2}
\, \eps^4 + o(\eps^4).
\end{align*}

\appendix

\section{Computation of the constants} \label{a:A}

\subsection{The Neumann-Dirichlet case}

In the present section, we use the above results to compute the quantities appearing in \cite[Section 4]{abatangelo2015sharp}. In order to avoid a conflict of notation with the present paper, for any odd positive integer $k$, we denote here by $\psi_k'$, $\mathfrak m_k'$ and $w_k'$ what is denoted in \cite{abatangelo2015sharp} by $\psi_k$, $\mathfrak m_k$ and $w_k$ respectively.

As in \cite{abatangelo2015sharp}, we use the notation
\begin{equation*}
	s_0:=\left\{\left(x'_1,0\right)\,;\,x'_1\ge 0\right\};
\end{equation*}
and
\begin{equation*}		
	s:=\left\{\left(x'_1,0\right)\,;\,x'_1\ge 1\right\}.
\end{equation*}
We now define the mapping $G:\RR^2_+\to \RR^2\setminus s_0$ by
\begin{equation*}
	G(x):=(x_1^2-x_2^2,2x_1x_2).
\end{equation*}
The mapping is conformal; indeed, if for $x\in\RR^2_+$ we write $z:=x_1+ix_2$ and $z':=x'_1+ix'_2$, with $x'=(x'_1,x'_2):=G(x)$, we have $z'=z^2$. The scale factor associated with $G$ is $h(x)=2|z|=2|x|$. Let  $u'$ be a function in $H^1\left(\RR^2\setminus s_0\right)$ and $u:=u'\circ G$. Since $G$ is conformal, $|\nabla u|$ is in $L^2\left(\RR^2_+\right)$, with 
\begin{equation*}
	\int_{\RR^2_+}\left|\nabla u\right|^2\,dx=\int_{\RR^2\setminus s_0}\left|\nabla u'\right|^2\,dx'.
\end{equation*}
Furthermore, for any $x'$ in the segment $(0,1)\times \{0\}$, which we
write as  $x'=(x_1',0)$, we have 
\begin{equation*}
	\frac{\partial u'}{\partial \nu_+}\left(x'\right)=-\frac1{2\sqrt{x'_1}}\frac{\partial  u}{\partial x_2}\left(\sqrt{x_1'},0\right)\mbox{ and }\frac{\partial u'}{\partial \nu_-}\left(x'\right)=-\frac1{2\sqrt{x'_1}}\frac{\partial  u}{\partial x_2}\left(-\sqrt{x_1'},0\right),
\end{equation*}
where $\frac{\partial u'}{\partial\nu_+}(x')$ and $\frac{\partial u'}{\partial\nu_-}(x')$ denote the normal derivative at $x'$ respectively from above and from below. We also note that $u$ is harmonic in $\RR^2_+$ if, and only if, $u'$ is harmonic in $\RR^2\setminus s_0$.

Let us now denote by $\widetilde{w}_k'$ the extension by reflexion to
$\RR^2\setminus s_0$ of $w_k'$, originally defined on
$\RR^2_+$. We recall that $w_k'$ is the unique finite
  energy solution to the problem 
\begin{equation*}
 \begin{cases}
   -\Delta w_k'=0, &\text{in }\RR^2_+, \\
   w_k'=0, &\text{on }s, \\
   \frac{\partial w_k'}{\partial \nu}=-\frac{\partial \psi'_k}{\partial
     \nu}, &\text{on }\partial \RR^2_+\setminus s,
 \end{cases}
\end{equation*}
where $ \psi_k'(r\cos t,r\sin t)= r^{k/2} \sin
  \big(\frac{k}{2}\,t\big)$.

\begin{lem} \label{lem2Prob} For any odd positive integer $k$, $w_k=\widetilde{w}_k'\circ G$.
\end{lem}

\begin{proof} Let us write $v:= \widetilde{w}_k'\circ G$. By
  uniqueness, it is enough to prove that $v$ solves
 \eqref{eq:wk}. From the remarks at the beginning of
  the present section, it follows that $v$ is harmonic in $\RR^2_+$.
  Let us now show that $\psi_k:=\psi_k'\circ G$. Indeed, for $x'\in
  \RR^2\setminus s_0$, $\psi_k'(x')=\mbox{Im}\left((z')^{k/2}\right)$,
  and therefore
  $f(x)=\mbox{Im}\left((z^2)^{k/2}\right)=\mbox{Im}\left(z^k\right)=\psi_k(x)$,
  where $x'=G(x)$, $z$ and $z'$ are defined as above, and where we use
  the determination of the square root on $\CC\setminus s_0$ defined
  by $G^{-1}$. From  this and the previous remarks, it follows that
  $v$ satisfies the boundary conditions of Problem \eqref{eq:wk}.
\end{proof}

As in  \cite{abatangelo2015sharp} we define
\[
\mathfrak m_k'=-\frac12\int_{\RR^2_+}\left|\nabla w_k'\right|^2\,dx
\]
and 
\begin{equation}\label{eq:mathfrak-m}
\mathfrak m_k=-\frac12\int_{\RR^2_+}\left|\nabla w_k\right|^2\,dx.
\end{equation}
We note that the right hand side of \eqref{eq:gad} is equal to $-2\beta^2\mathfrak m_k$.

\begin{cor}\label{cor:mk}	For any odd positive integer $k$, $\mathfrak m_k'=\frac12\mathfrak m_k$.
\end{cor}

\begin{proof} We have
\begin{equation*}
	\mathfrak m_k'=-\frac12\int_{\RR^2_+}\left|\nabla w_k'\right|^2\,dx'=-\frac14\int_{\RR^2\setminus s_0}\left|\nabla \widetilde{w}_k'\right|^2\,dx'.
\end{equation*} 
Using Lemma \ref{lem2Prob} and the conformal invariance of the $L^2$-norm of the gradient, we find
\begin{equation*}
	\int_{\RR^2\setminus s_0}\left|\nabla \widetilde{w}_k'\right|^2\,dx'=\int_{\RR^2_+}\left|\nabla w_k\right|^2\,dx=-2\mathfrak m_k. \qedhere
\end{equation*}
\end{proof}

In particular, Corollary \ref{cor:mk} and Proposition
  \ref{propmk} imply that 
\[
\mathfrak m_k'=
-\frac{k\pi}{4\,2^{2k-1}}\left(\begin{array}{c}k-1\\ \left\lfloor\frac{k-1}2\right\rfloor\end{array}\right)^{\!2},
\]
thus proving, in view of \cite[Theorem 1.2]{abatangelo2015sharp},
 the explicit constant appearing in the asymptotic
expansion of Theorem \ref{t:monopole}.

\subsection{The $u$-capacities of segments}
\label{subs:uCapacity}

In this last section, we simplify the constant $C_k$ occurring in \cite[Lemma 2.3]{AFHL2016}. 

\begin{prop}\label{p:A3} For any positive integer $k$, 
\begin{equation*}
	C_k=\frac{k}{4^{k-1}}\left(\begin{array}{c}k-1\\ \left\lfloor\frac{k-1}2\right\rfloor\end{array}\right)^2.
\end{equation*}
\end{prop}

\begin{proof} According to Equation (22) in \cite[Lemma 2.3]{AFHL2016}, 
\begin{equation*}
	C_k=\sum_{j=1}^k j\left|A_{j,k}\right|^2,
\end{equation*}
where $A_{j,k}$ is the $j$-th cosine Fourier coefficient of the function $\eta\mapsto (\cos\eta)^k$. To be more explicit, let us expand $(\cos\eta)^k$ into a trigonometric polynomial. We write 
\begin{equation*}
	(\cos\eta)^k=\left(\frac{e^{i\eta}+e^{-i\eta}}2\right)^k=\frac1{2^k}\sum_{j=0}^k\left(\begin{array}{c}k\\ j\end{array}\right)e^{(k-2j)i\eta}.
\end{equation*}
By grouping the terms of the sum in pairs starting from opposite extremities, we find
\begin{equation*}
	(\cos\eta)^k=\frac1{2^{k-1}}\sum_{j=0}^{\left\lfloor\frac{k-1}2\right\rfloor}\left(\begin{array}{c}k\\ j \end{array}\right)\cos((k-2j)\eta)+c_k
\end{equation*}
where 
\begin{equation*}
	c_k=0 \mbox{ if } k=2p+1 \mbox{\hspace{1cm} and \hspace{1cm}} c_k=\frac1{2^k}\left(\begin{array}{c}k\\ p \end{array}\right) \mbox{ if } k=2p.
\end{equation*}
It follows that 
\begin{equation*}
	C_k=\frac1{4^{k-1}}\sum_{j=0}^{\left\lfloor\frac{k-1}2\right\rfloor}(k-2j)\left(\begin{array}{c}k\\ j \end{array}\right)^2
\end{equation*}
and we conclude using Lemma \ref{lemSumBin}.
\end{proof}

Proposition \ref{p:A3} and
  \cite[Theorem 1.16]{AFHL2016} 
provide  the explicit constant appearing in the asymptotic
expansion of Theorem \ref{t:dipoleSymEven}.

\section{Auxiliary results for eigenvalues variations} \label{a:B}

This section is dedicated to the proof of Propositions \ref{p:ConvEven} and \ref{p:AsymptNDN}. In order to make a connection to the results of \cite{AFHL2016}, which we use, let us present an alternative characterization of the eigenvalues $(\mu_j)_{j\ge1}$ and $(\mu_j(\eps))_{j\ge1}$. We define 
\begin{equation*}
  \widehat{\mathcal Q}^s_\eps:=
\left\{u\in H^1_0(\widehat\Omega\setminus\Gamma_\eps)\,:\,u\circ\sigma=u\right\},
\end{equation*}
and we denote by $\widehat q^s_\eps$ the restriction of $\widehat q_0$
(see the paragraph preceding Theorem \ref{thmNonZero} for the
notation) to $\widehat{\mathcal Q}_\eps$. One can then check that we
obtain the eigenvalues $(\mu_j(\eps))_{j\ge1}$ from
$\widehat q^s_\eps$ by the min-max principle. In the same way, we
define
\begin{equation*}
	\widehat{\mathcal Q}^s:=
\left\{u\in H^1_0(\widehat\Omega)\,:\,u\circ\sigma=u\right\},
\end{equation*}
we denote by $\widehat q^s$ the restriction of the quadratic form
$\widehat q_0$, and one can check that we obtain the eigenvalues
$(\mu_j)_{j\ge1}$ from $\widehat q^s$ by the min-max principle. Let us
note that $-\Delta^s$, defined in Remark \ref{remDecompLap} as a
self-adjoint operator in $\mbox{ker}\left(I-\Sigma\right)$, is the
Friedrichs extension of $\widehat q^s$. We denote by $-\Delta^s_\eps$
the Friedrichs extension of $\widehat q^s_\eps$, which is also a
self-adjoint operator in $\mbox{ker}\left(I-\Sigma\right)$.  

Let us first prove Proposition \ref{p:ConvEven}. Since
$\mu_N(\eps)\ge\mu_N$ for all $\eps \in(0,\eps_0]$ and since
$\eps\mapsto\mu_N(\eps)$ is non-decreasing, we have existence of
$\mu_N^*:=\lim_{\eps\to0^+}\mu_N(\eps)$, with $\mu_N^*\ge\mu_N$. It
only remains to show that $\mu_N^*\le\mu_N$. In order to do this, let
us note that the space
\begin{equation*}
	\mathcal D^s:=\left\{u\in C^\infty_c(\widehat\Omega\setminus\{0\})\,:\, u=u\circ\sigma\ \right\}
\end{equation*}
is dense in $\mbox{ker}\left(I-\Sigma\right)$. Indeed, the space $C^\infty_c(\widehat\Omega\setminus\{0\})$
  is dense in $L^2(\widehat\Omega)$, since $\{0\}$ has measure
  $0$. Therefore, if we fix
$u\in \mbox{ker}\left(I-\Sigma\right)$, there exists a
sequence $(\varphi_n)_{n\ge1}$ of elements of
$C^\infty_c(\widehat\Omega\setminus\{0\})$ converging to $u$ in
$L^2(\widehat\Omega)$. We now set
$\widetilde \varphi_n:=1/2(\varphi_n+\varphi_n\circ\sigma)$. We have
$\widetilde\varphi_n\in\mathcal D^s$ for every integer $n\ge1$. Since
$u=1/2(u+u\circ\sigma)$, we have the inequality
\begin{equation*}
  \|\widetilde\varphi_n-u\|_{L^2(\widehat\Omega)}\le
  \frac12\|\varphi_n-u\|_{L^2(\widehat\Omega)}+
  \frac12\|\varphi_n\circ\sigma-u\circ\sigma\|_{L^2(\widehat\Omega)}
=\|\varphi_n-u\|_{L^2(\widehat\Omega)},
\end{equation*}
and this implies that the  sequence $(\widetilde\varphi_n)_{n\ge1}$ converges to $u$ in $\mbox{ker}\left(I-\Sigma\right)$.

According to the min-max characterization of eigenvalues and the previous density result, 
\begin{equation*}
	\mu_N=\inf_{\substack{
\mathcal E\subset \mathcal D^s\text{ subspace}\\
\mbox{dim}(\mathcal E)=N}}\max_{u\in \mathcal E} \frac{\widehat q_0(u)}{\|u\|^2}.
\end{equation*}
Let us now fix $\delta>0$ and an $N$-dimensional
  subspace $\mathcal E_\delta\subset \mathcal D^s$ such that
\begin{equation*}
	\max_{u\in \mathcal E_\delta} \frac{\widehat q_0(u)}{\|u\|^2}\le \mu_N+\delta.
\end{equation*}
There exists $\eps_1>0$ such that $\mathcal E_\delta\subset \widehat{\mathcal Q}_\eps^s$ for every $\eps\in(0,\eps_1]$. This implies that, for every $\eps\in(0,\eps_1]$,
\begin{equation*}
	\mu_N(\eps)=\min_{\substack{
\mathcal E\subset \widehat{\mathcal Q}_\eps^s\text{ subspace}\\
\mbox{dim}(\mathcal E)=N}}\max_{u\in \mathcal E} \frac{\widehat
q_\eps^s(u)}{\|u\|^2}
\le\max_{u\in \mathcal E_\delta} \frac{\widehat q_0(u)}{\|u\|^2}\le \mu_N+\delta.
\end{equation*}
Passing to the limit, we obtain first $\mu_N^*\le \mu_N+\delta$, and
then $\mu_N^*\le\mu_N$, concluding the proof.

\medskip
Let us finally prove Proposition \ref{p:AsymptNDN}. We recall that, as
a corollary of Theorem 1.10 in \cite{AFHL2016},
taking into account Proposition \ref{p:A3} we have the following result.
\begin{prop}\label{p:AFHL}
Let $\widehat\lambda_N$ be a simple eigenvalue of $-\widehat\Delta$ and $u_N$ an associated eigenfunction normalized in
$L^2(\widehat\Omega)$. Let us assume that $u_N\in
\widehat{\mathcal Q}^s$.
For $\eps>0$ small, we denote as $\widehat
  \lambda_N(\eps)$ the $N$-th eigenvalue of the Dirichlet Laplacian in
$\widehat\Omega\setminus \Gamma_\eps$.
 If $u_N(0)\neq0$, then
\begin{equation*}
	\widehat \lambda_N(\eps)=\widehat \lambda_N+\frac{2\pi}{|\log(\eps)|}u_N(0)^2+o\left(\frac1{|\log(\eps)|}\right) \quad \text{as $\eps\to0^+$}.
\end{equation*}
If  
\begin{equation*}
	r^{-k} u_N(r\cos t,r\sin t)\to \widehat\beta \cos\left(k t\right) \mbox{ in } C^{1,\tau}\left([0,\pi],\RR\right)
\end{equation*}
as $r\to 0^+$ for all $\tau\in (0,1)$, with $k\in \NN^*$ and $\widehat\beta\in \RR\setminus \{0\}$, then
\begin{equation*}
  \widehat \lambda_N(\eps)=\widehat
  \lambda_N+\frac{k\pi{\widehat\beta}^2}{4^{k-1}}\left(\begin{array}{c}
                                                         k-1\\
                                                         \left\lfloor\frac{k-1}2\right\rfloor\end{array}\right)^2\eps^{2k}+o\left(\eps^{2k}\right)
\quad \text{as $\eps\to0^+$}.
\end{equation*}
\end{prop}
Let us note that if the hypotheses of Proposition \ref{p:AFHL} are
satisfied, $\widehat\lambda_N$ is a simple eigenvalue of
$-\Delta^s$ and $u_N$ an associated eigenfunction. But the
converse is not true. Indeed, we have seen in Section
  \ref{subs:Disk}, in the case of $\widehat\lambda_3$ for the unit
  disk that $\widehat\lambda_N$ can be simple for $-\Delta^s$
  without being simple for $-\widehat\Delta$. Proposition
\ref{p:AFHL} is therefore weaker than Proposition
\ref{p:AsymptNDN}. However, the proof of Theorem 1.10 in
\cite{AFHL2016} can be adapted to prove Proposition
\ref{p:AsymptNDN}. Let us sketch the changes to be made.  The
proof in \cite{AFHL2016} mainly relies on Theorem 1.4
of \cite{AFHL2016}, and uses
the $u$-capacity and the associated potential defined in \cite[Equations
(6), (7), and (8)]{AFHL2016}. The following Lemma gives an alternative expression
when both $u$ and the compact set $K$ are symmetric; 
  it follows easily from Steiner symmetrization arguments.
\begin{lem} \label{l:SymCap} If $u\in\widehat Q^s$ and $K\subset
  \widehat \Omega$ is a compact set such that $\sigma(K)=K$, then
\begin{equation*}
	\mbox{Cap}_{\widehat\Omega}(K,u)=\min\left\{\widehat q^s(V)\,:\,V\in \widehat{\mathcal Q}^s\mbox{ and }u-V\in H^1_0(\widehat\Omega\setminus K)\right\}
\end{equation*}
and  the potential $V_{K,u}$ attaining the above minimum belongs to $\widehat{\mathcal Q}^s$.
\end{lem}
Our proof of Proposition \ref{p:AsymptNDN} relies on the following analog to \cite[Theorem 1.4]{AFHL2016}.
\begin{prop}\label{p:AsymCap}	Let $\mu_L$ be a simple eigenvalue of
  $-\Delta^s$ and $u_L$ an associated eigenfunction, normalized in $L^2(\widehat\Omega)$. Then 
\begin{equation*}
	\mu_L(\eps)=\mu_L+\mbox{Cap}_{\widehat\Omega}(\Gamma_\eps,u_L)+o\left(\mbox{Cap}_{\widehat\Omega}(\Gamma_\eps,u_L)\right) \quad \text{as $\eps\to0^+$}.
\end{equation*}
\end{prop}
In order to prove Proposition \ref{p:AsymCap}, we note that
 Lemma
\ref{l:SymCap} implies in particular that $u_L-V_{\Gamma_\eps,u_L}$ is the
orthogonal projection of $u_L$ on
$H^1_0(\widehat\Omega\setminus \Gamma_\eps)\cap \widehat{\mathcal Q}^s$ and
$\mbox{Cap}_{\widehat \Omega}(\Gamma_\eps,u_L)$ the square of the distance of $u_L$
from $H^1_0(\widehat\Omega\setminus \Gamma_\eps)\cap \widehat{\mathcal Q}^s$,
both defined with respect to the scalar product induced by
$\widehat q^s$ on $\widehat{\mathcal Q}^s$. We also note that we can
use the estimates of $V_{\Gamma_\eps,u_L}$ given in Lemma A.1 and Corollary A.2 of
\cite{AFHL2016}. We can therefore repeat step by step the proof of
Theorem 1.4 in Appendix A of \cite{AFHL2016}, replacing 
  $L^2(\widehat\Omega)$ by $\mbox{ker}\left(I-\Sigma\right)$,
$H^1_0(\widehat \Omega)$ with $\widehat{\mathcal Q}^s$,
$H^1_0(\widehat\Omega\setminus \Gamma_\eps)$ by
$H^1_0(\widehat\Omega\setminus \Gamma_\eps)\cap \widehat{\mathcal Q}^s$,
$\widehat q$ and $\widehat q_\eps$ by $\widehat q^s$ and
$\widehat q_\eps^s$,  $-\widehat\Delta$ and
  $-\widehat\Delta_\eps$ by $-\Delta^s$ and $-\Delta_\eps^s$,
$\widehat\lambda_N$ by $\mu_L$ and $u_N\in H^1_0(\widehat \Omega)$ by
$u_L\in \widehat{\mathcal Q}^s$. We obtain Proposition
\ref{p:AsymCap}.  The estimates of $\mbox{Cap}_{\widehat\Omega}(\Gamma_\eps,u)$
proved in \cite[Section 2]{AFHL2016} then give us Proposition
\ref{p:AsymptNDN}.

\section{Alternative proof of Theorem \ref{t:gad}}\label{sec:altern-proof-theor}

We find useful to show an alternative proof of Theorem \ref{t:gad}. 
This proof is based on 
sharp estimates from above and below of the 
Rayleigh quotients for the eigenvalues $\lambda_N$ and
$\lambda_N(\eps)$. Such estimates require energy bounds  on
eigenfunctions 
obtained by an Almgren type
monotonicity argument and blow-up analysis for scaled
eigenfunctions. We mention that  such a strategy was first developed
in  \cite{abatangelo2015sharp,AbatangeloFelli2016SIAM,AbatangeloFelliNorisNys2016,NNT} for eigenvalues of Aharonov--Bohm
operators with a moving pole.
On the other hand,
 the
implementation of this procedure for our problem requires a
quite different technique with respect to the 
case of  Aharonov--Bohm operators with a single pole,
when estimating a singular term appearing in the
derivate of the Almgren frequency function (i.e. the term
\eqref{eq:Meps}). Indeed, in the single pole case estimates can be
derived by 
rewriting the problem as a Laplace equation on the twofold covering, whereas
in this case the singular term
\eqref{eq:Meps} turns out to have a negative sign and this is enough
to proceed with the monotonicity argument (see Subsection \ref{s:monotonicity}).

In this argument, an important step is a blow-up result for scaled eigenfunctions.
%
%

In what follows, we aim at pointing out the main steps of the proof, together with a more deepened analysis at the crucial points. 
We list below some notation used throughout
this appendix.\par
\begin{itemize}
\item[-] For  $r>0$ and $a\in\RR^2$, $D_r(a)=\{x\in\RR^2:|x-a|<r\}$
  denotes the disk of center $a$ and radius $r$. We also denote the
  corresponding upper half-disk as
  $D_r^+(a)=\{(x_1,x_2)\in D_r(a):x_2>0\}$.
\item[-] For all $r>0$, 
$D_r=D_r(0)$ is the disk of center $0$ and
radius $r$; $D_r^+=\{(x_1,x_2)\in D_r:x_2>0\}$ denotes the
  corresponding upper half-disk.
\item[-] For  $r>0$ and $a\in\RR^2$,
  $S_r^+(a)=\{(x_1,x_2)\in\partial D_r(a):x_2>0\}$ denotes the upper
  half-circle of center $a$ and radius $r$. We also denote $S_r^+:=S_r^+(0)$.
\end{itemize}

\subsection{Limit profile}

This section contains a variational construction of the limit
profile which will be used to describe the limit of the  blow-up sequence.

Let us consider  the functional $J_k: \mathcal Q\to\RR$ (see Subsection \ref{subsec:relres} for the definition of $\mathcal Q$)
\begin{equation}\label{eq:Jk}
J_k(u) = \frac12 \int_{\RR^2_+} |\nabla u(x)|^2 \,dx-
 \int_{-1}^1 u(x_1,0)\frac{\partial \psi_k}{\partial x_2}(x_1,0)\,dx_1,
\end{equation}
with $\psi_k$ defined in \eqref{eq:psi_k}. We observe that
$\frac{\partial \psi_k}{\partial x_2}(x_1,0) = k {x_1}^{k-1}$ and
$J_k$ is well-defined on $\mathcal Q$.

\begin{lem}\label{l:wk}
For all $k\in\NN$, $k\geq1$, 
let $w_k\in \mathcal Q$ be the unique weak solution to  \eqref{eq:wk}
and let $\mathfrak m_k=-\frac12\int_{\RR^2_+}\left|\nabla
  w_k\right|^2\,dx$ be as in \eqref{eq:mathfrak-m}. Then
\begin{equation}\label{eq:Ik}
\mathfrak m_k=\min_{u\in\mathcal Q}J_k(u)=J_k(w_k)<0.
\end{equation}
 Furthermore, $w_k(x)=O\big(\tfrac{1}{|x|}\big)$ 
as $|x|\to+\infty$. 
\end{lem}
\begin{proof}
The proof follows from standard minimization methods, Hardy Inequality and Kelvin Transform. 
\end{proof}

\begin{lem}\label{l:Phi}
 For every $k\in \NN$, $k\geq1$, there exists a unique
$\Phi_k \in \bigcap_{R>0}H^1(D_R^+)$  such that 
\begin{equation}\label{eq:3}
\begin{cases}
\Phi_k-\psi_k\in \mathcal Q,\\
  -\Delta \Phi_k =0, &\text{in } \RR^2_+  \text{ in a
    distributional sense},\\
\Phi_k =0 &\text{on } s,\\  
\frac{\partial \Phi_k}{\partial \nu}=0&\text{on }\Gamma_1,
\end{cases}
\end{equation}
where $\nu=(0,-1)$ is the outer normal unit vector on $\partial
\RR^2_+$.
Furthermore, the unique solution to \eqref{eq:3} is given by 
\[
\Phi_k=\psi_k + w_k,
\]
 where $w_k$ is as in Lemma \ref{l:wk} and $\psi_k$ is defined in \eqref{eq:psi_k}.
\end{lem}
\begin{proof}
 The existence part is proved by taking $\Phi_k=\psi_k + w_k$.
  To prove uniqueness, one can argue by contradiction exploiting the
  Hardy Inequality
(see\cite[Proposition 4.3]{abatangelo2015sharp} for a detailed
proof in a similar problem).
\end{proof}

For future convenience, we state and prove here the following lemma, 
which relates the limit profile $\Phi_k$ (more precisely, its $k$-th 
Fourier coefficient) to the minimum $\mathfrak{m}_k$.
\begin{lem}\label{l:xi1}
 Let $\Phi_k$ be as in Lemma \ref{l:Phi}. Then 
 \[
  \int_0^\pi \Phi_k(\cos t,\sin t) \,\sin(kt)\,dt = -\frac{\mathfrak{m}_k}k + \frac\pi2.
 \]
\end{lem}
\begin{proof}
 Let us define the function 
 \[
  \omega(r):= \int_0^\pi w_k(r\cos t,r\sin t)\,\sin(kt)\,dt,\quad r>0,
 \]
where $w_k$ is as in Lemma \ref{l:wk}.
Then, recalling that $\Phi_k= w_k+\psi_k$, we have that 
\begin{equation}\label{eq:omega1primopasso}
 \omega(1)= \int_0^\pi \Phi_k(\cos t,\sin t) \,\sin(kt)\,dt - \frac\pi2.
\end{equation}
 Since $\omega$ is the $k$-th Fourier coefficient of the harmonic
  function $w_k$, it satisfies the
differential  equation
$\omega''+\frac1r\omega'-\frac{k^2}{r^2}\omega=0$ in $(1,+\infty)$,
i.e. $(r^{1+2k}(r^{-k} \omega)')'=0$. Hence there exists
$C_\omega\in\RR$ such that 
$\big( r^{-k}\omega(r) \big)'= C_\omega r^{-(1+2k)}$, for $r>1$.
Integrating the previous equation over $(1,r)$ we obtain that
\[
\dfrac{\omega(r)}{r^{k}} - \omega(1) = \frac{C_\omega}{2k}\left(
  1-\frac1{r^{2k}} \right),\quad\text{for all }r\geq1.
 \]
Lemma \ref{l:wk} provides that $\omega(r)=O(r^{-1})$ as
 $r\to+\infty$, hence, letting $r\to+\infty$ in the previous identity,
 we obtain that necessarily $C_\omega=-2k\omega(1)$ and then
\begin{equation}\label{eq:120}
 \omega(r)= \omega(1) r^{-k},\quad 
 \omega'(r)= -k\omega(1) r^{-k -1},\quad\text{for all }r\geq1.
\end{equation}
On the other hand, by definition 
\begin{equation}\label{eq:121}
  \omega'(r)= {r^{-k-1}} \int_{ S_r^+}
  \frac{\partial w_k}{\partial \nu}\,\psi_k\,ds,
\end{equation} 
with $\nu$ being the outer unit
vector to $\partial D_r^+$.
Combining \eqref{eq:120} and \eqref{eq:121} we obtain that 
\begin{equation*}
  \omega(1)=-\dfrac{1}{k} \int_{S_1^+} \dfrac{\partial w_k}{\partial \nu}\,\psi_k\,ds.
\end{equation*}
 Multiplying the equation $-\Delta w_k=0$ by $\psi_k$, integrating by parts on
$D_1^+$,  and recalling that $\psi_k\equiv0$ on $\Gamma_1$, we obtain that
\begin{equation*}
 \int_{D_1^+} \nabla w_k \cdot \nabla \psi_k\,dx=\int_{\partial D_1^+}
 \dfrac{\partial w_k}{\partial \nu}\,\psi_k\,ds = 
\int_{S_1^+}
 \dfrac{\partial w_k}{\partial \nu}\,\psi_k\,ds,
\end{equation*}
whereas multiplying $-\Delta\psi_k=0$ by $w_k$ and integrating by parts on
$D_1^+$ we obtain that
\begin{equation*}
  \int_{D_1^+} \nabla w_k \cdot \nabla \psi_k\,dx=\int_{\partial
    D_1^+} 
\dfrac{\partial \psi_k}{\partial \nu}\,w_k\,ds.
\end{equation*}
Taking into account the boundary data, we obtain that 
\[
 \int_{S_1^+} \dfrac{\partial w_k}{\partial\nu}\,\psi_k =   
 \int_{S_1^+} \dfrac{\partial \psi_k}{\partial\nu}\,w_k 
 - \int_{\Gamma_1} \dfrac{\partial\psi_k}{\partial x_2}\,w_k, 
\]
so that 
\begin{equation}\label{eq:omega1}
 \omega(1)= -\frac1k \int_{S_1^+} \dfrac{\partial \psi_k}{\partial\nu}\,w_k
 + \dfrac1k \int_{\Gamma_1} \dfrac{\partial \psi_k}{\partial x_2} \,w_k.
\end{equation}
 Since $\frac{\partial\psi_k}{\partial\nu}=k\psi_k$ on
  $S_1^+$, it results that $k\omega(1)=\int_{S_1^+} \frac{\partial\psi_k}{\partial\nu}w_k$,
so that \eqref{eq:omega1} can be rewritten as
$\omega(1)=-\omega(1) +  \frac{1}k \int_{\Gamma_1} \frac{\partial \psi_k}{\partial x_2} \,w_k$ and thus
\begin{equation*}
  \omega(1)= {\frac{1}{2k}} \int_{\Gamma_1} \dfrac{\partial \psi_k}{\partial x_2}\,w_k .
\end{equation*}
From \eqref{eq:Ik} we deduce that 
$\omega(1)= -\frac{1}{k} {\mathfrak m}_k$, 
and recalling \eqref{eq:omega1primopasso} the proof is concluded.
\end{proof}

\subsection{Monotonicity argument}\label{s:monotonicity}

In order to prove convergence of blow-up eigenfunctions, energy
estimates in small neighborhoods of the Dirichlet-Neumann junctions
are needed; such estimates are obtained via an Almgren type monotonicity argument which is sketched here. 

For $\lambda \in \RR$,  $u \in H^{1}(\Omega)$ and 
$r\in(0,\eps_0)$ such that $D_r^+\subset\Omega$, the Almgren frequency function is defined as
\[
\mathcal{N}(u,r,\lambda) = \dfrac{E(u,r,\lambda)}{H(u,r)},
\]
where
\begin{equation*}
E(u,r,\lambda) = \int_{D_r^+} \Big( |\nabla u(x)|^2 - \lambda
                 u^2(x) \Big) \,dx, \quad
H(u,r)  = \dfrac1r \int_{S_r^+} u^2\,ds .
\end{equation*}
In the following, we assume that assumption \eqref{eq:6} is satisfied, i.e.  the $N$-th
eigenvalue  $\lambda_N$ of $q_0$ is simple, and we fix an
  associated normalized eigenfunction $u_N$, so that $u_N$ satisfies \eqref{eq:5}.
For all $1\leq n< N$, let
$u_n\in H^1_0(\Omega)$ be an eigenfunction of $q_0$ associated to the
eigenvalue $\lambda_n$ such that
\begin{equation*}
  \int_\Omega |u_n(x)|^2\,dx=1\quad \text{for all }1\leq n<N
\end{equation*}
and 
\begin{equation*}
  \int_\Omega u_n(x) u_m(x)\,dx=0\quad \text{if }1\leq
  n,m\leq N\text{ and }n\neq m.
\end{equation*}
  For every 
  $\eps\in (0,\eps_0]$, let $u_N^{\eps}$ be an
  eigenfunction of $q_{\eps}$ associated with $\lambda_N(\eps)$,
  i.e. solving
\begin{equation}\label{eqDNDeps}
\begin{cases}
-\Delta\,u_N^{\eps}=\lambda_N(\eps)\,u_N^{\eps},& \mbox{in }\Omega,\\
u_N^{\eps}=0,&\mbox{on } \partial\Omega\setminus \Gamma_{\eps},\\[5pt]
\dfrac{\partial u_N^{\eps}}{\partial \nu}=0,&\mbox{on } \Gamma_{\eps},
\end{cases}
\end{equation}
 such
  that 
  \begin{equation}\label{eq:7}
\int_\Omega |u_N^{\eps}(x)|^2\,dx=1\quad\text{and}\quad 
\int_{\Omega}u_N^{\eps}(x)\,u_N (x)\,dx\ge0.
\end{equation}
For all $1\leq n< N$ and $\eps\in (0,\eps_0]$, let
$u_n^\eps\in \mathcal Q_{\eps}$ be an eigenfunction of problem
\eqref{eqDND} associated to the eigenvalue $\lambda=\lambda_n(\eps)$,
i.e. solving
\begin{equation}\label{eq:equneps}
\begin{cases}
-\Delta\,u_n^{\eps}=\lambda_n(\eps)\,u_n^{\eps},& \mbox{in }\Omega,\\
u_n^{\eps}=0,&\mbox{on } \partial\Omega\setminus \Gamma_{\eps},\\[5pt]
\dfrac{\partial u_n^{\eps}}{\partial \nu}=0,&\mbox{on } \Gamma_{\eps},
\end{cases}
\end{equation}
such that 
\begin{equation}\label{eq:8}
  \int_\Omega |u_n^\eps(x)|^2\,dx=1\quad \text{for all }1\leq n<N
\end{equation}
and 
\begin{equation}\label{eq:14}
  \int_\Omega u_n^\eps(x) u_m^\eps(x)\,dx=0\quad \text{if }1\leq
  n,m\leq N\text{ and }n\neq m.
\end{equation}
 We observe that, in view of Remark \ref{remIneq},
\begin{equation}\label{eq:bound_l_neps}
\lambda_n(\eps)\leq \lambda_N\quad\text{for all
}\eps\in(0,\eps_0]\text{ and }1\leq n\leq N.
\end{equation} 
Arguing as in \cite[Lemma 5.2]{abatangelo2015sharp},
it is possible to prove the following properties:
\begin{enumerate}[\rm (i)]
\item there exists $R_0\in\big(0,\min\big\{\eps_0,\frac1{2\sqrt{\lambda_N}}\big\}\big)$ such that
  $D^+_{R_0}\subset\Omega$ and 
\begin{equation*}
  H(u_n^\eps,r)>0\quad\text{for all }\eps\in(0,R_0),\  r\in(\eps,R_0) \text{ and }1\leq
  n\leq N;
\end{equation*}
\item 
\label{prop--ii} for every $r\in(0,R_0]$, there exist $C_r > 0$ and $\alpha_r \in
(0,r)$ such that $H(u_n^\eps,r)
 \geq C_r$ for all $\eps\in(0, \alpha_r)$ and $1 \leq n\leq N$.
  \end{enumerate}
By direct calculations it follows that,
for all $\eps\in(0,R_0)$, $\eps<r<R_0$, and $n\in\{1,2,\dots,N\}$,
\begin{align}\label{eq:derH}
& \dfrac{d}{dr} H(u_n^\eps,r) 
= \dfrac2r\int_{S_r^+}u_n^\eps \frac{\partial u_n^\eps}{\partial\nu}\,ds
= \dfrac2r E(u_n^\eps,r,\lambda_n(\eps)), \\ 
 \label{eq:derE}& \dfrac{d}{dr} E(u_n^\eps,r,\lambda_n(\eps))=
 = 2\int_{S_r^+} \abs{\frac{\partial u_n^\eps}{\partial\nu}}^2\,ds- \frac2r \left( M(\eps,u_n^\eps,\lambda_n(\eps)) + \lambda_n(\eps) \int_{D_r^+} (u_n^\eps(x))^2\,dx \right)
 \end{align}
where 
$\nu$ denotes the exterior normal unit vector to $D_r^+$ and 
    \begin{equation}\label{eq:Meps}
     M(\eps,u,\lambda) = \lim_{\delta\to0^+} 
     \int_{\RR^2_+\cap\partial A_\delta^\eps}\bigg( \frac12 |\nabla u|^2 x \cdot \boldsymbol n
     -  \frac{\partial u}{\partial \boldsymbol n}(x\cdot\nabla u)
     - \frac{\lambda}{2}  u^2 x\cdot \boldsymbol n\bigg)ds,
    \end{equation}
being $A_\delta^{\eps} := D ^+_\delta (-\eps,0) \cup D_\delta^+(\eps,0)$
and $\boldsymbol n$ denoting  the exterior normal unit vector to $D_r^+\setminus
  A_\delta^\eps$. For details in a similar problem see \cite [Lemma 5.5 and
  5.6]{NNT}. 
A crucial step in the monotonicity argument is the possibility of
recognizing the sign of the quantity $M(\eps,u,\lambda)$.

To this aim, we first state the following
result describing the behaviour of solutions  to \eqref{eqDND}
at  Dirichlet-Neumann boundary  junctions.
\begin{prop}\label{p:asymptotics}
Let $\eps \in (0,\eps_0)$, $\lambda\in \RR$, and $u \in
\mathcal{Q}_\eps\setminus\{0\}$
 be a nontrivial solution to problem \eqref{eqDND}.
Then there exist two odd natural numbers
$j_L=j_L(\eps,u,\lambda),j_R =j_R(\eps,u,\lambda)\in \NN$ and two nonzero real
numbers $\beta_L=\beta_L(\eps,u,\lambda),\beta_R=\beta_R(\eps,u,\lambda) \in \RR\setminus\{0\}$ such that
\begin{align}
  &\delta^{-{j_L}/2} u((-\eps,0) + \delta \boldsymbol \theta(t)) 
    \to \beta_L
\cos\big(\tfrac{j_L}{2}t\big)
    \quad \text{ in } C^{1,\sigma}([0,\pi]), \label{eq:asy-eigen-left}\\
  &\delta^{-{j_R}/2} u((\eps,0) + \delta\boldsymbol \theta(t)) \to
    \beta_R
 \sin\big(\tfrac{j_R}{2}t\big)
    \quad \text{ in } C^{1,\sigma}([0,\pi]),\label{eq:asy-eigen-right}
\end{align}
as $\delta\to0^+$ for any $\sigma\in (0,1)$, where $\boldsymbol \theta(t) = (\cos t,\sin t)$.
Moreover, 
\begin{align}
  &\delta^{-{j_L}/2+1} \nabla u((-\eps,0) + \delta
\boldsymbol \theta(t)) \to \tfrac{j_L \beta_L }{2}
    \left(\cos\big(\tfrac{j_L}{2}t\big)
    \boldsymbol\theta(t) - \sin\big(\tfrac{j_L}{2}t\big)
\boldsymbol\tau(t)\right)\label{eq:asy-eigen-grad-left}\\
  &\delta^{-{j_R}/2+1} \nabla u((\eps,0) + \delta
    \boldsymbol \theta(t)) \to  
\tfrac{j_R \beta_R }{2}
    \left(\sin\big(\tfrac{j_R}{2}t\big)\boldsymbol\theta(t) +
 \cos\big(\tfrac{j_R}{2}t\big)\boldsymbol\tau(t) \right)\label{eq:asy-eigen-grad-right}
\end{align}
in $C^{0,\sigma}([0,\pi])$
as $\delta\to0^+$ for any $\sigma\in (0,1)$, 
where $\boldsymbol \tau(t)=(-\sin t,\cos t)$.
\end{prop}
\begin{proof}
Through a gauge transformation, in a neighbourhood of each junction $(\pm \eps,0)$ the problem can be
rewritten as an elliptic equation with an Aharonov--Bohm vector
potential with pole located at the junction; then the asymptotics
follows from \cite[Theorem 1.3]{FFT}.
\end{proof}

\begin{lem}\label{l:segnoM}
 Let $\eps\in(0,\eps_0]$ and $u\in \mathcal{Q}_\eps$ be a solution to
 \eqref{eqDND} for some $\lambda\in\RR$. 
 Moreover, let $j_L=j_L(\eps,u,\lambda),\,j_R=j_R(\eps,u,\lambda)\in
 \NN$ odd and 
$\beta_L=\beta_L(\eps,u,\lambda)$, $\beta_R=\beta_R(\eps,u,\lambda) \in \RR\setminus\{0\}$
be  as in Proposition \ref{p:asymptotics} and let $M(\eps,u,\lambda)$
be as in \eqref{eq:Meps}. Then
 \begin{equation*}\label{eq:Meps2}
  M(\eps,u,\lambda)=
  \begin{cases}
   0, &\text{if }j_L>1 \text{ and }j_R >1,\\
   -\eps\frac{\pi}{8} \beta_L^2, &\text{if }j_L =1 \text{ and }j_R >1,\\
   -\eps\frac{\pi}{8}\beta_R^2, &\text{if }j_L >1 \text{ and }j_R =1,\\
   -\eps\frac{\pi}{8} \big(\beta_L^2 + \beta_R^2 \big), &\text{if }j_L =1 \text{ and }j_R=1.
  \end{cases}
 \end{equation*}
In particular, $M(\eps,u,\lambda)\leq0$.
\end{lem}
\begin{proof}
 Since $\partial A_\delta^\eps\cap \RR^2_+ = S_\delta^+(-\eps,0) \cup S_\delta^+(\eps,0)$, we split  \eqref{eq:Meps} into the corresponding two contributions. 
 
\smallskip\noindent  {\bf Negligible terms.} 
 On $S_\delta^+(-\eps,0)$, we have that $x=(-\eps,0)+\delta
 \boldsymbol\theta(t)$ for some $t\in[0,\pi]$
and $\boldsymbol n =-\boldsymbol\theta$,
where $\boldsymbol\theta(t)=(\cos
 t,\sin t )$; hence 
$x\cdot \boldsymbol n = \eps \cos t - \delta$. 
 From \eqref{eq:asy-eigen-left} and \eqref{eq:asy-eigen-grad-left} we have
 that $u((-\eps,0)+\delta
 \boldsymbol\theta(t))\to 0$ 
and 
 $|\nabla u((-\eps,0)+\delta
 \boldsymbol\theta(t))|^2 = \frac{j_L^2\beta_L^2}4\delta^{j_L-2}(1+o(1))$ uniformly on $[0,\pi]$ as $\delta\to0$.
 From the Dominated Convergence Theorem we then obtain 
 \begin{align*}
  \int_{S_\delta^+(-\eps,0)}& (|\nabla u|^2 
  -{\lambda}u^2)\,x\cdot \boldsymbol n\,ds\\
& =\delta\!\int_0^\pi\!\!
 (|\nabla u((-\eps,0)+\delta
 \boldsymbol\theta(t))|^2 
  -{\lambda}|u((-\eps,0)+\delta
 \boldsymbol\theta(t))|^2)\,(\eps \cos t - \delta)
\,dt\\
&  \to 
 \begin{cases}
 0, &\text{if }j_L>1,\\
 \frac{\beta_L^2\eps}4\int_0^\pi \cos t \,dt=0, &\text{if }j_L=1,
 \end{cases}
 \end{align*}
 as $\delta\to0$.

\smallskip
\noindent  {\bf Leading term.} 
 We now look at the last term 
 \[
 - \int_{S_\delta^+(-\eps,0)} \frac{\partial u}{\partial \boldsymbol n}(x\cdot\nabla u)\,ds
 = \int_{S_\delta^+(-\eps,0)}(\boldsymbol \theta \cdot\nabla u)(x\cdot\nabla u)\,ds,
 \]
 since $\boldsymbol \theta =-\boldsymbol n$ on $S_\delta^+(-\eps,0)$.
 From \eqref{eq:asy-eigen-grad-left} we have  
  \[
   \delta^{-j_L /2 +1}\nabla u((-\eps,0)+\delta
 \boldsymbol\theta(t)) \cdot \boldsymbol \theta(t) 
   \to \frac{j_L}{2} \beta_L \cos\bigg(\frac{j_L}{2} t\bigg)
  \]
  in $C^0([0,\pi])$ as $\delta\to0$. On the other hand, for
  $x=(-\eps,0)+\delta\boldsymbol \theta(t)$ we have that
 \[
  \delta^{-\frac{j_L}2 +1}\nabla u((-\eps,0)+\delta
 \boldsymbol\theta(t)) \cdot x 
  \to -\eps \frac{j_L}{2} \beta_L \left( \cos\Big(\frac{j_L}{2} t\Big) \cos t 
  + \sin\Big(\frac{j_L}{2} t\Big) \sin t\right)
 \]
  in $C^0([0,\pi])$ as $\delta\to0$.
 Thus, by the Dominated Convergence Theorem, we have 
\begin{align*}
 & - \int_{S_\delta^+(-\eps,0)} \frac{\partial u}{\partial
   \boldsymbol n}(x\cdot\nabla u)\,ds\\
&=\delta\!\int_0^\pi \!\!\!\left(\nabla u((-\eps,0)+\delta\boldsymbol
  \theta(t))\!\cdot\!((-\eps,0)+\delta\boldsymbol \theta(t))\right)
\left(\nabla u((-\eps,0)+\delta\boldsymbol
  \theta(t))\!\cdot\!\boldsymbol\theta(t)\right)dt\\ 
&\to
  \begin{cases}
    0,&\text{if }j_L>1,\\
-\frac{\eps }{4}(\beta_L)^2 \int_0^\pi \left(  \cos^2\big(\frac{t}{2} \big) \cos t + \cos\big(\frac{t}{2} \big) \sin\big(\frac{t}{2} \big)\sin t \right)dt ,&\text{if }j_L =1,
  \end{cases}\\
&=
  \begin{cases}
    0,&\text{if }j_L >1,\\
-\frac{\eps }{4}\beta_L^2 \int_0^\pi  \cos^2\big(\frac{t}{2} \big)\,dt ,&\text{if }j_L =1,
  \end{cases}\\
&= 
  \begin{cases}
   0, &\text{if }j_L >1,\\
   -\frac{\eps}{8} \beta_L^2 \pi, &\text{if }j_L =1,
  \end{cases}
 \end{align*}
as $\delta\to0$.
One can follow the same argument to compute the contribution coming from $S_\delta^+ (\eps,0)$. 
Putting together the two contributions we obtain the thesis. 
\end{proof}

This turns out to be sufficient to prove the following:
\begin{lem}\label{l:limitatezza_N_per_blowup}
For any $n\in\{1,\ldots,N\}$, 
$\eps\in(0,R_0)$,  and $r,R$ such that $\eps <r<R\leq R_0$ we have that 
 \[
  \mathcal{N}(u_n^\eps,r,\lambda_n(\eps)) + 1 
  \leq \left(\mathcal{N}(u_n^\eps,R,\lambda_n(\eps)) + 1\right)
  e^{2\lambda_N R^2}.
 \]
In particular,  for every $\delta \in (0, 1)$
there exists $r_\delta \!\in( 0,R_0)$ such that, 
 for any  $\eps\in(0,r_\delta)$ and $r\in(\eps,r_\delta)$,
$\mathcal{N}(u_N^\eps, r, \lambda_N(\eps)) \leq k +\delta$,
 $k$ being as in \eqref{eq:orderk}.
\end{lem}
\begin{proof}
Once the negative sign of $M(\eps,u,\lambda)$ is established
(Lemma \ref{l:segnoM}) the proof proceeds as in \cite[Section
5]{abatangelo2015sharp}. 
\end{proof}

Lemma \ref{l:limitatezza_N_per_blowup} is the key point for a priori estimates on energy of the blow up sequence in half-disks. These estimates are in turn fundamental to deduce estimates on the difference of eigenvalues, as it appears in the following subsection.

\subsection{Estimates on the difference of eigenvalues}\label{sec:C3}

Firstly,  we are going to estimate the Rayleigh quotient for $\lambda_N(\eps)$. 
 Let $R> 1$. With $R_0$ as in the previous section, for every $\eps\in (0,\eps_0)$
 such that $R\eps< R_0$ we define the function
 \begin{equation*}
v_{R,\eps} = 
\begin{cases}
v_{R,\eps}^{int}, &\text{in } D_{R\eps}^+, \\[3pt]
u_N, &\text{in } \Omega \setminus D_{R\eps}^+,
\end{cases}
\end{equation*}
where
$v_{R,\eps}^{int}$ is the unique solution to
\begin{equation}\label{eq:eqvint}
\begin{cases}
-\Delta v_{R,\eps}^{int} = 0, &\text{in } D_{R\eps}^+, \\[3pt]
     v_{R,\eps}^{int} = u_N, &\text{on } S_{R\eps}^+, \\[3pt]
     v_{R,\eps}^{int} = 0, &\text{on } \Gamma_{R\eps}\setminus\Gamma_\eps, \\[3pt]
     \frac{\partial v_{R,\eps}^{int}}{\partial \nu}=0, &\text{on } \Gamma_\eps,
   \end{cases}
\end{equation}
i.e., by the Dirichlet principle, the unique solution to the minimization problem 
\begin{equation}\label{eq:19-1}
 \int_{D_{R\eps}^+} |\nabla v_{R,\eps}^{int}|^2 dx= \min \left\{{\textstyle{ \int_{D_{R\eps}^+}}} |\nabla
   v|^2:
 v\in H^1(D_{R\eps}^+),\, v= u_N \text{ on }S_{R\eps}^+,
 \, v=0 \text{ on }\Gamma_{R\eps}\setminus \Gamma_\eps\right\}.
\end{equation}
In order to handle the denominator of the Rayleigh quotient we proceed with a Gram-Schmidt process. 
Since we are taking into account $u_1,\ldots,u_{N-1}$ as the first $N-1$ test functions for the Rayleigh quotient,
which are already orthonormalized in $L^2(\Omega)$, we define
\[
\tilde u_{R,\eps}= \frac{v_{R,\eps} - \sum_{j=1}^{N-1} \left(\int_\Omega v_{R,\eps}\,u_j\right) u_j}
{\nor{v_{R,\eps} - \sum_{j=1}^{N-1} \left(\int_\Omega v_{R,\eps}\,u_j\right) u_j}}_{L^2(\Omega)} .
\]
Using the Dirichlet Principle and the asymptotics
\eqref{eq:orderk} one can easily prove the following energy estimates
for  $v_{R,\eps}^{int}$ in small disks. 
\begin{lem}\label{l:estimatesv}
There exists a constant $C>0$ (independent of $\eps$ and $R$) such
that, for every $R>1$ and $\eps\in(0,\eps_0)$ such that $R\eps< R_0$,  the following estimates hold:
 \begin{align}
  \int_{D_{R\eps}^+} |\nabla v_{R,\eps}^{int}|^2dx \leq C(R\eps)^{2k},\label{eq:energyvint}\\
  \int_{S_{R\eps}^+} | v_{R,\eps}^{int}|^2ds\leq C(R\eps)^{2k+1},\label{eq:tracevint}\\
  \int_{D_{R\eps}^+} | v_{R,\eps}^{int}|^2 dx\leq C(R\eps)^{2k+2}.\label{eq:L2vint}
 \end{align}
\end{lem}
To our aim, for every $R>1$ we define
$v_R$ as the unique solution to the minimization problem 
\begin{equation*}
 \int_{D_R^+} |\nabla v_R|^2dx = \min \left\{ \int_{D_R^+} |\nabla
   v|^2dx:
 v\in H^1(D_R^+),\ v= \psi_k \text{ on }S_R^+,
 \ v=0 \text{ on }\Gamma_R\setminus \Gamma_1\right\}.
\end{equation*}
The function $v_R$ is the unique weak solution to
\begin{equation}\label{eq:vR}
 \begin{cases}
  -\Delta v_R =0, &\text{in }D_R^+,\\
  v_R= \psi_k, &\text{ on }S_R^+,\\
  v_R=0, &\text{ on }\Gamma_R\setminus \Gamma_1,\\
\frac{\partial v_{R}}{\partial \nu}=0, &\text{on } \Gamma_1.
 \end{cases}
\end{equation}
As well, we introduce the following blow-up functions
\begin{equation}\label{eq:25}
 U_\eps(x):= \frac{u_N(\eps x)}{\eps^k}, \quad V_\eps^R(x):=
 \frac{v_{R,\eps}^{int}(\eps x)}{\eps^k}.
\end{equation}
 Combining \eqref{eq:orderk} with the Dirichlet Principle,  
we can establish the following convergences
\begin{align}
 U_\eps \to \beta\psi_k \text{ as }\eps\to0 \text{ in
 }H^1(D_R^+)\text{  for every $R>1$};\label{eq:convergence1}\\
 V_\eps^R \to \beta v_R \text{ for $\eps\to0$ and for any $R>1$};\label{eq:convergence2}\\
 v_R\to\Phi_k \text{ in $H^{1}(D_r^+)$ as $R\to+\infty$ for any $r>1$.}\label{eq:convergence3}  
\end{align}

\begin{prop}\label{p:stimafacile}
For any $R>1$ and $\eps\in(0,\eps_0)$ such that $R\eps<R_0$, we have that
 \[
  \frac{\lambda_N(\eps) - \lambda_N}{\eps^{2k}} \leq f_R(\eps)  
 \]
 where 
 \[
\lim_{\eps\to 0 } f_R(\eps) =\beta^2 \int_{S_R^+} \psi_k
 \bigg(\frac{\partial v_R}{\partial \nu} - \frac{\partial \psi_k}{\partial\nu} \bigg)\,ds
 \]
 with $\psi_k$ defined in \eqref{eq:psi_k} and $v_R$ in \eqref{eq:vR}. 
\end{prop}
\begin{proof}
We note that 
\begin{align}
 \notag&\bigg\|v_{R,\eps} - \sum_{j=1}^{N-1} \left(\int_\Omega v_{R,\eps}\,u_j\,dx\right) u_j\bigg\|_{L^2(\Omega)}^2= \nor{v_{R,\eps}}_{L^2(\Omega)}^2 - \sum_{j=1}^{N-1} \bigg(\int_\Omega v_{R,\eps}\,u_j\,dx\bigg)^{\!\!2}\\
 \notag&= 1 - \int_{D_{R\eps}^+} u_N^2\,dx + \int_{D^+_{R\eps}} |
         v_{R,\eps}^{int}|^2\,dx  - \sum_{j=1}^{N-1} 
 \bigg(\int_\Omega v_{R,\eps}\,u_j\,dx\bigg)^{\!\!2}\\
 &= 1 + O(\eps^{2k+2})\quad\text{as }\eps\to 0 \label{eq:L2gramschmidt}
\end{align}
in view of \eqref{eq:orderk} and \eqref{eq:L2vint} and since, for all
$j<N$,
\begin{equation}\label{eq:L2mixedGramSchmidt}
 \int_\Omega v_{R,\eps}\,u_j\,dx
 = -\int_{D_{R\eps}^+} u_N\,u_j\,dx + \int_{D_{R\eps}^+}
 v_{R,\eps}^{int}\,u_j\,dx = O(\eps^{k+2})\text{ as }\eps\to0.
\end{equation}
The functions $u_1,\ldots,u_{N-1},\tilde u_{R,\eps}$ are linearly
independent (since they are nontrivial and mutually
  orthogonal) and belong to $\mathcal Q_\eps$; 
if we plug a linear combination of them into the Rayleigh quotient \eqref{eqMinMaxDND}
we obtain
\begin{align*}
&\lambda_N(\eps)-\lambda_N     \leq \left(\max_
  {\substack{(\alpha_1,\dots, \alpha_{N})\in \RR^{N}\\
  \sum_{j=1}^{N}|\alpha_j|^2 =1}}
  \int_\Omega \bigg| \nabla \bigg( \sum_{j=1}^{N-1}\alpha_j u_j +
  \alpha_N \tilde u_{R,\eps} \bigg) \bigg|^2
\right)-\lambda_N\\  
   &= \max_{\substack{(\alpha_1,\dots, \alpha_{N})\in \RR^{N}\\   \sum_{j=1}^{N}|\alpha_j|^2 =1}}
\!\!\left[  \sum_{j=1}^{N-1}\alpha_j^2\lambda_j + \alpha_N^2 \int_\Omega\!| \nabla \tilde u_{R,\eps} |^2
    + 2 \sum_{j=1}^{N-1}{\alpha_j} \alpha_N \int_\Omega\!\nabla u_j\! \cdot\! \nabla \tilde u_{R,\eps}\!-\!\lambda_N\!\right]\\  
   &= \max_{\substack{(\alpha_1,\dots, \alpha_{N})\in \RR^{N}\\   \sum_{j=1}^{N}|\alpha_j|^2 =1}}
   \Bigg[\sum_{j=1}^{N-1}\alpha_j^2 (\lambda_j - \lambda_N) + \alpha_N^2 \left( \int_\Omega | \nabla \tilde u_{R,\eps} |^2 -\lambda_N\right) + 2 \sum_{j=1}^{N-1}{\alpha_j} \alpha_N \int_\Omega \nabla u_j \cdot \nabla \tilde u_{R,\eps} \Bigg].
\end{align*}
In view of \eqref{eq:energyvint} and \eqref{eq:orderk} we
  have that 
\begin{equation}\label{eq:9}
 \int_\Omega \nabla v_{R,\eps}\cdot \nabla u_j = 
-\int_{D_{R\eps}^+} \nabla u_N\cdot \nabla u_j\,dx + \int_{D_{R\eps}^+}
 \nabla v_{R,\eps}^{int}\cdot\nabla u_j\,dx=
O(\eps^{k+1}) .
\end{equation}
Moreover, from 
convergences \eqref{eq:convergence1}--\eqref{eq:convergence3} we have
\begin{align}
\notag  &\int_\Omega |\nabla v_{R,\eps}|^2 dx-\lambda_N 
\notag = -\int_{D_{R\eps}^+} |\nabla u_N|^2 dx +  \int_{D_{R\eps}^+}
  |\nabla v_{R,\eps}^{int}|^2dx \\
\notag&= \eps^{2k}\bigg(-\int_{D_{R}^+} |\nabla U_\eps|^2 dx +  \int_{D_{R}^+}
  |\nabla V_\eps^R|^2dx \bigg)= \eps^{2k}\beta^2\bigg(-\int_{D_{R}^+} |\nabla \psi_k|^2 dx +  \int_{D_{R}^+}
  |\nabla v_R|^2dx+o(1)\bigg)\\
\label{eq:10}&   = \eps^{2k} \beta^2\left( \int_{S_R^+} \psi_k \bigg(\frac{\partial v_R}{\partial \nu} - \frac{\partial \psi_k}{\partial\nu} \bigg) ds
   + o(1)\right) \quad \text{as }\eps\to0. 
\end{align}
Collecting \eqref{eq:L2gramschmidt},
  \eqref{eq:L2mixedGramSchmidt}, \eqref{eq:9}, and \eqref{eq:10},
we obtain that 
\begin{align*}
& \int_\Omega | \nabla \tilde u_{R,\eps} |^2 -\lambda_N \\
 &= \frac{\int_\Omega |\nabla v_{R,\eps}|^2 + \sum_{j=1}\limits^{N-1} \!\left(\int_\Omega v_{R,\eps}\,u_j\right)^2 \!\lambda_j 
 - 2\sum\limits_{j=1}^{N-1} \left(\int_\Omega v_{R,\eps}\,u_j\right) \int_\Omega \nabla v_{R,\eps}\cdot \nabla u_j}
 {\Big\|v_{R,\eps} - \sum_{j=1}^{N-1} \left(\int_\Omega v_{R,\eps}\,u_j\right) u_j\Big\|_{L^2(\Omega)}^2} - \lambda_N\\
 &= \eps^{2k} \beta^2\left( \int_{S_R^+} \psi_k \bigg(\frac{\partial v_R}{\partial \nu} - \frac{\partial \psi_k}{\partial\nu} \bigg) ds
   + o(1)\right) \quad \text{as }\eps\to0.
\end{align*}
From
  \eqref{eq:L2gramschmidt}, \eqref{eq:L2mixedGramSchmidt}, and
  \eqref{eq:9} it follows that, for every $j<N$, 
  \begin{align*}
     \int_\Omega \nabla u_j \cdot \nabla \tilde u_{R,\eps} =O(\eps^{k+1}) \quad \text{as }\eps\to0.
  \end{align*}
Hence, the assumptions in \cite[Lemma 6.1]{abatangelo2015sharp} are fulfilled by
 $\mu(\eps)= 
\int_{S_R^+} \psi_k \big(\frac{\partial v_R}{\partial \nu} - \frac{\partial \psi_k}{\partial\nu} \big) ds
   + o(1)$, $\alpha=1$, $\sigma(\eps)= \beta^2\eps^{2k}$ and $M=2k-1$ and the conclusion follows.
\end{proof}

In the sequel we denote 
\begin{equation}\label{eq:11}
 \kappa_R := \int_{S_R^+} \psi_k \bigg(\frac{\partial v_R}{\partial \nu} - \frac{\partial \psi_k}{\partial\nu} \bigg)\,ds.
\end{equation}
\begin{lem}\label{l:limkappaR}
Let $\kappa_R$ be defined in \eqref{eq:11}. Then
$\lim_{R\to+\infty} \kappa_R = 2\,\mathfrak{m}_k$,
with $\mathfrak{m}_k$ as in \eqref{eq:Ik}.
\end{lem}
\begin{proof}
From \eqref{eq:vR} it
follows that the function $\sigma_R$ defined as 
\begin{equation}\label{eq:upsilon_R}
\sigma_R(r):= 
\int_0^{\pi}  v_R(r(\cos t,\sin t)) \sin(kt) \,dt,\quad r\in[1,R],
\end{equation}
satisfies the equation
  $(r^{1+2k}(r^{-k}\sigma_R)')'=0$ and hence, for some $c_{R}\in\RR$, $\big( r^{-k}\sigma_R(r) \big)'=
\frac{c_{R}}{r^{1+2k}}$ in  $(1,R)$. 
Integrating the previous equation over $(1,r)$ we obtain 
\begin{equation}\label{eq:125}
 r^{-k} \sigma_R(r) -\sigma_R(1) =
\frac{c_{R}}{2k}\bigg(1-\frac1{r^{2k}}\bigg), \quad \text{for all }r\in(1,R]. 
\end{equation}
Since \eqref{eq:vR} implies that 
$\sigma_R(R)=\frac12\pi R^{k}$,
from \eqref{eq:125} we deduce that 
\[
\frac{c_R}{2k}=\frac{R^{2k}}{R^{2k}-1}\bigg(\frac\pi2-\sigma_R(1)\bigg)
\]
and then 
\begin{align*}
  \sigma_R(r) 
  =r^{k} \frac{\tfrac\pi2 R^{2k}-\sigma_R(1)}{R^{2k}-1}-r^{-k}\frac{R^{2k}}{R^{2k}-1}(\tfrac\pi2-\sigma_R(1)),
\end{align*}
for all $r\in(1,R]$. If we differentiate the previous identity and 
evaluate it in $r=R$, we obtain 
\begin{equation}\label{eq:127}
\sigma_R'(R)=k\,\frac{R^{k-1}}{R^{2k}-1}\Big(\tfrac\pi2(R^{2k}+1)-2\sigma_R(1)\Big).
\end{equation}
On the other hand, 
differentiating \eqref{eq:upsilon_R}, we obtain that
\begin{equation}\label{eq:128}
  \sigma_R'(r)= r^{-1-k} \int_{S_r^+} \nabla v_R\cdot \nu\, \psi_k\,ds
\end{equation}
and then from \eqref{eq:127} and \eqref{eq:128}
\begin{equation}\label{eq:12}
 \sigma_R'(R)= R^{-1-k} \int_{S_R^+} \psi_k\, \frac{\partial v_R}{\partial\nu}\,ds
 = k\frac{R^{k-1}}{R^{2k}-1} \Big(\tfrac\pi2 (R^{2k}+1) - 2\sigma_R(1) \Big) .
\end{equation}
As well, from the definition of $\psi_k$ \eqref{eq:psi_k} we
have that
\begin{equation}\label{eq:13}
 \int_{S_R^+} \psi_k\,\frac{\partial \psi_k}{\partial\nu}\,ds = \frac\pi2 k\,R^{2k}.
\end{equation}
 Combining \eqref{eq:12} and \eqref{eq:13} we obtain that 
\[
\kappa_R=
\frac{2k\, R^{2k}}{R^{2k}-1} \bigg(\frac\pi2-\sigma_R(1)\bigg)
=
\frac{2k\, R^{2k}}{R^{2k}-1} \bigg(\frac\pi2-\int_0^{\pi}  v_R(\cos t,\sin t) \sin(kt) \,dt\bigg)
\]
and hence, via \eqref{eq:convergence3},
\[
\lim_{R\to+\infty} \kappa_R =
2k \bigg(\frac\pi2-\int_0^{\pi}  \Phi_k(\cos t,\sin t) \sin(kt) \,dt\bigg).
\]
By Lemma \ref{l:xi1}, the proof is concluded.
\end{proof}

We are now going to estimate the Rayleigh quotient for $\lambda_N$. 
 Let $R\geq 1$. Choosing $R_0$ as in the previous subsection, for every $\eps\in (0,\eps_0)$
 such that $R\eps< R_0$ and for any $j=1,\ldots,N$ we define the function
 \begin{equation}\label{eq:24}
w_{j,R,\eps} = 
\begin{cases}
w_{j,R,\eps}^{int}, &\text{in } D_{R\eps}^+, \\[3pt]
w_{j,R,\eps}^{ext}, &\text{in } \Omega \setminus D_{R\eps}^+,
\end{cases}
\end{equation}
where, letting $u_j^\eps$ be as in \eqref{eqDNDeps}--\eqref{eq:14},
\[
w_{j,R,\eps}^{ext} = u_j^\eps
\quad \text{ in } \Omega \setminus D_{R\eps}^+,
\]
and $w_{j,R,\eps}^{int}$ is the unique solution to
\begin{equation}\label{eq:eqwint}
\begin{cases}
-\Delta w_{j,R,\eps}^{int} = 0, &\text{in } D_{R\eps}^+, \\[3pt]
     w_{j,R,\eps}^{int} = u_j^\eps, &\text{on } S_{R\eps}^+, \\[3pt]
     w_{j,R,\eps}^{int} = 0, &\text{on } \Gamma_{R\eps}.
   \end{cases}
\end{equation}
By the Dirichlet principle, we have that
  $w_{j,R,\eps}^{int}$ is  the unique solution to the minimization problem 
\begin{equation}\label{eq:19}
 \int_{D_{R\eps}^+} |\nabla w_{j,R,\eps}^{int}|^2 dx= \min \left\{{\textstyle{ \int_{D_{R\eps}^+}}} |\nabla
   v|^2dx:
 v\in H^1(D_{R\eps}^+),\, v= u_j^\eps\text{ on }S_{R\eps}^+,
 \, v=0 \text{ on }\Gamma_{R\eps}\right\}.
\end{equation}
In order to handle the denominator we proceed with a Gram-Schmidt process. 
We then define 
\begin{equation}\label{eq:hatuj} 
\hat u_{j,R,\eps}:= \dfrac{\tilde w_{j,R,\eps}}{\|\tilde w_{j,R,\eps}\|_{L^2(\Omega)}}, \quad j=1,\ldots, N, 
\end{equation}
where $\tilde w_{N,R,\eps} := w_{N,R,\eps}$ and 
\begin{equation*}
\tilde w_{j,R,\eps} := w_{j,R,\eps} - \sum_{\ell= j+1}^{N}\dfrac{\int_\Omega
  w_{j,R,\eps} {\tilde w_{\ell,R,\eps}}\,dx}{\|\tilde w_{\ell,R,\eps}\|_{L^2(\Omega)}^2} \tilde w_{\ell,R,\eps} 
\quad \text{for }j=1,\ldots, N-1.
\end{equation*}
We can derive the following estimate of
the energy of eigenfunctions $u_j^\eps$ in half-disks of radius of order $\eps$.

\begin{lem}\label{l:sec_con}
  For $1\leq j\leq N$ and $\eps\in(0,\eps_0)$, let $u_j^\eps$ be as in
  \eqref{eqDNDeps}--\eqref{eq:14}.  
  For every $\delta\in(0,1/2)$, there exists $\mu_\delta>1$
such that, for all $R\geq \mu_\delta$,
  $\eps<\frac{R_0}R$, and $1\leq j\leq N$,
\begin{align}
\label{eq:34}&\int_{S^+_{R\eps}}|u_j^\eps|^2\,ds\leq C (R\eps)^{3-2\delta},\\
\label{eq:35}&\int_{D^+_{R\eps}}|\nabla u_j^\eps|^2\,dx\leq
C (R\eps)^{2-2\delta},\\
\label{eq:36}&\int_{D^+_{R\eps}}|u_j^\eps|^2\,dx\leq
C (R\eps)^{4-2\delta},\\
\label{eq:L2tracewjint}&\int_{S^+_{R\eps}}|w_{j,R,\eps}^{int}|^2\,ds\leq C (R\eps)^{3-2\delta},\\
\label{eq:L2gradwjint}&\int_{D^+_{R\eps}}|\nabla w_{j,R,\eps}^{int}|^2\,dx\leq
C (R\eps)^{2-2\delta},\\
\label{eq:L2wjint}&\int_{D^+_{R\eps}}|w_{j,R,\eps}^{int}|^2\,dx\leq
C (R\eps)^{4-2\delta},
\end{align}
for some constant $C>0$ depending only on $R_0$ and $\lambda_N$.
\end{lem}
\begin{proof}
From \eqref{eq:equneps} and \eqref{eq:bound_l_neps} we know
  that $\{u_j^\eps\}_{\eps\in(0,\eps_0)}$ is bounded in $H^1$; hence,
  from of property (ii) at page \pageref{prop--ii} we deduce that, for
  $\eps$ sufficiently small, $N(u_j^\eps, R_0,\lambda_j(\eps))$ is
  bounded uniformly with respect to $\eps$. Estimates
  \eqref{eq:34}--\eqref{eq:36} then follow from Lemma
  \ref{l:limitatezza_N_per_blowup}; we refer to \cite[Lemma
  5.8]{abatangelo2015sharp} for a detailed proof in a similar
  problem. Estimates \eqref{eq:L2tracewjint}--\eqref{eq:L2wjint} can
  be proved combining estimates \eqref{eq:34}--\eqref{eq:36}  with
  the Dirichlet principle (see \cite[Lemma 6.2]{abatangelo2015sharp} for details in a similar
  problem).
\end{proof}

For $\delta\in(0,1/2)$ fixed, let $\mu_\delta$ be as in Lemma \ref{l:sec_con}.
For $\eps$ sufficiently small in such a way that $\mu_\delta\eps<R_0$, we introduce the following blow-up functions:
\begin{equation}\label{eq:blowup}
 \hat U_\eps(x):= \frac{u_N^\eps(\eps x)}{\sqrt{H(u_N^\eps,\mu_\delta\eps)}}, \qquad 
 W_\eps^R(x):= \frac{w_{N,R,\eps}^{int}(\eps x)}{\sqrt{H(u_N^\eps,\mu_\delta\eps)}}.
\end{equation}
We notice that, by scaling,
\begin{equation}\label{eq:26}
  \frac{1}{\mu_\delta}\int_{S_{\mu_\delta}^+}|\hat U_\eps|^2\,ds=1.
\end{equation}

\begin{thm}\label{t:stime_blowup}
Let  $\delta\in(0,1/2)$ be fixed and let $r_\delta>0$ be as in Lemma
\ref{l:limitatezza_N_per_blowup}.
For all $R\geq \mu_\delta$,  
\begin{equation}\label{eq:blowupbounded}
\text{the family of functions }\big\{ \hat U_\eps(x):\ R\eps<r_\delta  \big\}
 \text{ is bounded in $H^{1}(D_R^+)$}.
\end{equation}
In particular, for all $R\geq \mu_\delta$,  
\begin{align}
&\int_{D_{R\eps}^+} \abs{\nabla u_N^\eps}^2dx=O(H(u_N^\eps,\mu_\delta\eps)),\quad\text{as }\eps\to0^+,\label{eq:L2graduNint}\\
&\int_{S_{R\eps}^+} |u_N^\eps|^2ds=O(\eps H(u_N^\eps,\mu_\delta\eps)),\quad\text{as }\eps\to0^+,\\
&\int_{D_{R\eps}^+} |u_N^\eps|^2dx=O(\eps^2 H(u_N^\eps,\mu_\delta\eps)),\quad\text{as }\eps\to0^+. \label{eq:L2uNint} 
\end{align}
\end{thm}
\begin{proof}We omit the proof which can be derived from the
  monotonicity result given in  Lemma
  \ref{l:limitatezza_N_per_blowup} following the same
  argument as Lemma \ref{l:sec_con}; for details in an analogous
  problem  we refer to
  \cite[Theorem 5.9]{abatangelo2015sharp}.  
\end{proof}

By the Dirichlet principle and Theorem \ref{t:stime_blowup} we have
also the following estimates.
\begin{lem}\label{l:stimeWepsR}
For all $R>\max\{2,\mu_\delta\}$,  
\begin{equation}\label{eq:WepsRbounded}
\text{the family of functions }\big\{W_\eps^R: R\eps<{r_\delta}\big\}
 \text{ is bounded in $H^{1}(D_R^+)$}.
\end{equation}
In particular, for all $R>\max\{2,\mu_\delta\}$,  
\begin{align}
&\int_{D_{R\eps}^+} \abs{\nabla w_{N,R,\eps}^{int}}^2dx=O(H(u_N^\eps,\mu_\delta\eps)),\quad\text{as }\eps\to0^+,\label{eq:L2gradwNint}\\
\label{eq:21}&\int_{S_{R\eps}^+} |w_{N,R,\eps}^{int}|^2dx=O(\eps H(u_N^\eps,\mu_\delta\eps)),\quad\text{as }\eps\to0^+,\\
\label{eq:22}&\int_{D_{R\eps}^+} |w_{N,R,\eps}^{int}|^2dx=O(\eps^2 H(u_N^\eps,\mu_\delta\eps)),\quad\text{as }\eps\to0^+.  
\end{align}
\end{lem}

 We are now in position to prove a sharp upper bound for the
eigenvalue variation $\lambda_N-\lambda_N(\eps)$.
\begin{prop}\label{p:stimadifficile}
 There exists
  $\tilde R>2$ such that, for all $R>\tilde R$ and  $\eps>0$ with
  $R\eps<{R_0}$, 
\[
\frac{\lambda_N-\lambda_N(\eps)}{H(u_N^\eps, \mu_\delta\eps)}\leq g_R(\eps)
\]
where 
\begin{align}
g_R(\eps)=
\int_{D_R^+}|\nabla W_\eps^R|^2\,dx-\int_{D_R^+}|\nabla \hat
  U_\eps|^2\,dx+o(1)
\quad\text{and}\quad
g_R(\eps)=O(1) 
\quad\text{as }\eps\to 0^+, \label{eq:gR}
\end{align}
with $\hat U_\eps$ and  $W_\eps^R$ defined in \eqref{eq:blowup}.
\end{prop}
\begin{proof}
 As already mentioned, we take into account the Courant--Fisher
 characterization for $\lambda_N$ 
recalled in \eqref{eqMinMaxD}  and consider the $N$-dimensional space spanned by the functions $\{\hat u_{j,R,\eps,}\}_{j=1}^N$
 defined in \eqref{eq:hatuj}. Before proceeding, we note that 
 \begin{align}
&  \| \tilde w_{N,R,\eps} \|_{L^2(\Omega)}^2 = 1 + O\big(\eps^2 H(u_N^\eps,\mu_\delta\eps)\big),  \notag \\
 & d_{N,j}^{R,\eps}:= \dfrac{\int_\Omega  w_{j,R,\eps} {\tilde w_{N,R,\eps}}\,dx}{\|\tilde w_{N,R,\eps}\|_{L^2(\Omega)}^2} 
  = O\big(\eps^{3-\delta} \sqrt{H(u_N^\eps,\mu_\delta\eps)} \big),
    \quad\text{for all }j< N, \label{eq:stimeNterm}
 \end{align}
as $\eps\to0$, thanks to \eqref{eq:L2uNint}, \eqref{eq:22}, \eqref{eq:36} and \eqref{eq:L2wjint}. On the other hand, 
\begin{align}
&  \| \tilde w_{j,R,\eps} \|_{L^2(\Omega)}^2 = 1 + O(\eps^{4-2\delta}), \notag \\
  &d_{\ell,j}^{R,\eps}:= \dfrac{\int_\Omega  w_{j,R,\eps} {\tilde w_{\ell,R,\eps}}\,dx}{\|\tilde w_{\ell,R,\eps}\|_{L^2(\Omega)}^2} 
  = O(\eps^{4-2\delta}), \quad\text{for all }j\neq \ell, \label{eq:stimejterm}
 \end{align}
 as $\eps\to 0$ and
for any $j=1,\ldots,N-1$ thanks to \eqref{eq:L2wjint} and \eqref{eq:36}.

If we plug a linear combination of $\{\hat u_{j,R,\eps,}\}_{j=1}^N$
into the Rayleigh quotient we obtain 
that
\begin{equation*}
  \lambda_N \leq \max_
  {\substack{(\alpha_1,\dots, \alpha_{N})\in \RR^{N}\\
      \sum_{j=1}^{N}|\alpha_j|^2 =1}}
  \int_{\Omega} \bigg|\nabla  \bigg(\sum_{j=1}^{N}\alpha_j
     \hat u_{j,R,\eps}\bigg)\bigg|^2 dx,
\end{equation*}
and then
\begin{equation}\label{eq:107}
\lambda_N-\lambda_N(\eps)\leq \max_
  {\substack{(\alpha_1,\dots, \alpha_N)\in \RR^N\\
      \sum_{j=1}^{N}|\alpha_j|^2 =1}}\sum_{j,n=1}^N
m_{j,n}^{\eps,R}\alpha_j {\alpha_n},
\end{equation}
where 
\[
m_{j,n}^{\eps,R}= \int_{\Omega} \nabla    \hat u_{j,R,\eps}\cdot \nabla \hat u_{n,R,\eps}\, dx
-\lambda_N(\eps) \delta_{jn},
\]
with $\delta_{jn}=1$ if $j=n$ and 
$\delta_{jn}=0$ if $j\neq n$.

From \eqref{eq:derH} and Lemma \ref{l:limitatezza_N_per_blowup}, if
$R\geq \mu_\delta$ and $R\eps<{r_\delta}$ we have
\begin{equation}
  \frac1 {H(u_N^\eps,r)}\dfrac{d}{dr} H(u_N^\eps,r) = \dfrac2r
  \mathcal N(u_N^\eps, r,\lambda_N(\eps))\leq \frac
  2r({k}+\delta)\quad\text{for all }\mu_\delta\eps\leq r\leq r_\delta,\label{eq:23}
\end{equation}
Integration of \eqref{eq:23} 
over the interval $(\mu_\delta\eps, r_\delta)$
and  property (ii) at page \pageref{prop--ii} 
 yield
\begin{equation}\label{eq:stima_sotto_radiceH}
H(u_N^\eps, \mu_\delta\eps) \geq C_\delta \eps^{2k + 2\delta},\quad\text{if }\mu_\delta\eps<{r_\delta},
\end{equation}
for some $C_\delta>0$ independent of $\eps$. 
Estimate
\eqref{eq:34} implies that  
\begin{equation}\label{eq:stima_sopra_radiceH}
 H(u_N^\eps, \mu_\delta\eps)=O(\eps^{2-2\delta})\quad\text{ as }\eps\to0.
\end{equation}
From \eqref{eq:stimeNterm},
\eqref{eq:blowup},  Theorem \ref{t:stime_blowup}, and
Lemma \ref{l:stimeWepsR} we deduce that 
\begin{align}
  m_{N,N}^{\eps,R} &= \dfrac{\lambda_N(\eps) (1-\|w_{N,R,\eps}
    \|_{L^2(\Omega)}^2)}{\|w_{N,R,\eps} \|_{L^2(\Omega)}^2} +
  \dfrac{\left( \int_{D_{R\eps}^+} \big| \nabla w_{N,R,\eps}^{int}\big|^2 dx
   - \int_{D_{R\eps}^+} \big| \nabla u_N^\eps \big|^2 dx\right)}{\|w_{N,R,\eps} \|_{L^2(\Omega)}^2}\notag \\
 \label{eq:mNN} &= H(u_N^\eps,\mu_\delta\eps)\bigg(\int_{D_R^+}|\nabla W_\eps^R|^2\,dx-\int_{D_R^+}|\nabla \hat U_\eps|^2\,dx+o(1)\bigg),
\end{align}
as $\eps\to0^+$.
On the other hand, if $j<N$, by the convergence of the perturbed eigenvalue,
\eqref{eq:stimejterm}, \eqref{eq:L2gradwjint}, \eqref{eq:35} 
we have that \begin{align*}
  m_{j,j}^{\eps,R}&= -\lambda_N(\eps) 
  +\dfrac{1}{\|\tilde w_{j,R,\eps}\|_{L^2(\Omega)}^2} \left(
    \lambda_j(\eps) - \int_{D_{R\eps}^+}\!\! \!\abs{\nabla u_j^\eps}^2dx
    + \int_{D_{R\eps}^+}\! \!\!\abs{\nabla w_{j,R,\eps}^{int}}^2dx \right) \\
  &\ + \dfrac{1}{\|\tilde w_{j,R,\eps}\|_{L^2(\Omega)}^2}
  \int_{\Omega} \bigg|\nabla 
  \Big(\sum_{\ell>j} d_{\ell,j}^{R,\eps} \tilde w_{\ell,R,\eps} \Big)\bigg|^2dx\\
  & \ -\dfrac{2}{\|\tilde w_{j,R,\eps}\|_{L^2(\Omega)}^2}
  \bigg(\int_{\Omega} \nabla w_{j,R,\eps} \cdot
  {\nabla \Big( \sum_{\ell>j} d_{\ell,j}^{R,\eps} \tilde w_{\ell,R,\eps} \Big)}\,dx\bigg)\\
  &=(\lambda_j-\lambda_N)+o(1)\quad\text{as }\eps\to0.
\end{align*}
From \eqref{eq:stimeNterm}, \eqref{eq:stimejterm}, \eqref{eq:35}, 
\eqref{eq:L2gradwjint}, \eqref{eq:L2graduNint}, and 
\eqref{eq:L2gradwNint}, it follows that, for all $j<N$,
\begin{align*}
\|\tilde w_{j,R,\eps}\|_{L^2(\Omega)}& \|\tilde
w_{N,R,\eps}\|_{L^2(\Omega)}  m_{j,N}^{\eps,R}= \int_{D_{R\eps}^+}\Big(\nabla
    w_{j,R,\eps}^{int}\cdot {\nabla w_{N,R,\eps}^{int} }-
    \nabla u_j^\eps\cdot
    {\nabla u_N^\eps}\Big)\,dx \\
 &\quad - \int_{\Omega} \nabla
    \Big(\sum_{\ell>j}d_{\ell,j}^{R,\eps} \tilde w_{\ell,R,\eps}\Big) \cdot
    {\nabla w_{N,R,\eps} }\,dx
=O\Big(\eps^{1-\delta}
  \sqrt{H(u_N^\eps,\mu_\delta\eps})\Big).
\end{align*}
Hence, by \eqref{eq:stimeNterm} and \eqref{eq:stimejterm}, we have that 
\[
  m_{j,N}^{\eps,R}=O\Big(\eps^{1-\delta}
  \sqrt{H(u_N^\eps,\mu_\delta\eps})\Big)\quad\text{and}\quad 
m_{N,j}^{\eps,R}={ m_{j,N}^{\eps,R}}=O\Big(\eps^{1-\delta} \sqrt{H(u_N^\eps,\mu_\delta\eps})\Big)
\]
as 
$\eps\to 0^+$.
From  \eqref{eq:stimejterm},
\eqref{eq:35}, 
\eqref{eq:L2gradwjint},
we deduce that, for all $j,n<N$ with $j\neq n$,
\begin{align*}
   \|\tilde w_{j,R,\eps}&\|_{L^2(\Omega)}\|\tilde w_{n,R,\eps}\|_{L^2(\Omega)}  m_{j,n}^{\eps,R}\\
&=\int_{D_{R\eps}^+}\Big(\nabla w_{j,R,\eps}^{int}\cdot {\nabla w_{n,R,\eps}^{int} }-
    \nabla u_j^\eps\cdot {\nabla
      u_n^\eps}\Big)\,dx\\
  &\quad + \int_{\Omega} \nabla \big( \sum_{\ell>j}
    d_{\ell,j}^{R,\eps} \tilde w_{\ell,R,\eps} \big) \cdot {\nabla \big( \sum_{h>n} d_{h,n}^{R,\eps} \tilde w_{h,R,\eps} \big)}\,dx
  \\
  &\quad - \int_{\Omega} \nabla \big( \sum_{\ell>j}
    d_{\ell,j}^{R,\eps} \tilde w_{\ell,R,\eps} \big) \cdot {\nabla w_{n,R,\eps}} \,dx
  \\
  &\quad - \int_{\Omega} \nabla w_{j,R,\eps} \cdot
{\nabla \big( \sum_{h>n} d_{h,n}^{R,\eps} \tilde w_{h,R,\eps} \big)}\,dx=O(\eps^{2-2\delta}) \quad\text{as }\eps\to0.
\end{align*}
Hence, in view of \eqref{eq:stimejterm}, 
\[
  m_{j,n}^{\eps,R}=O(\eps^{2-2\delta}) \quad\text{as }\eps\to0.
\]
Taking into account \eqref{eq:stima_sotto_radiceH}, we can then apply
 \cite[Lemma 6.1]{abatangelo2015sharp}
 with $\sigma(\eps)=H(u_N^\eps,\mu_\delta\eps)$,
$\mu(\eps)=g_R(\eps)$, $\alpha=1-\delta$ and $M=4k$ in order to
deduce \[ \max_
{\substack{(\alpha_1,\dots, \alpha_{N})\in \RR^{N}\\
    \sum_{j=1}^{N}|\alpha_j|^2 =1}}\sum_{j,n=1}^N
m_{j,n}^{\eps,R}\alpha_j {\alpha_n} =H(u_N^\eps,\mu_\delta\eps)\bigg(
\int_{D_R^+}|\nabla W_\eps^R|^2\,dx-\int_{D_R^+}|\nabla \hat
U_\eps|^2\,dx+o(1)\bigg)
\]
as $\eps\to 0^+$,
which, in view of \eqref{eq:107}, yields
$\frac{\lambda_N-\lambda_N(\eps)}{H(u_N^\eps,\mu_\delta\eps)}\leq g_R(\eps)$
with $g_R$ as in \eqref{eq:gR}.
We notice that, from Theorem \ref{t:stime_blowup} and Lemma
\ref{l:stimeWepsR}, for all $R>\max\{2,\mu_\delta\}$, $g_R(\eps)=O(1)$ as $\eps\to 0^+$.
The proof is now complete.
\end{proof}

Combining Proposition \ref{p:stimafacile}, Lemma \ref{l:limkappaR} and
Proposition \ref{p:stimadifficile} we obtain the following upper/lower
estimates for $\lambda_N-\lambda_N(\eps)$.

\begin{prop}\label{prop:quasi_finito1}
There exists a positive
  constant $C^*>0$ such that
\[
-2\beta^2 \mathfrak{m}_k \,\eps^{2k} (1+o(1))
\leq \lambda_N - \lambda_N(\eps)
\leq C^*\, H(u_N^\eps, \mu_\delta\eps), 
\quad \text{as }\eps\to 0,
\]
with $\mathfrak{m}_k<0$ as in \eqref{eq:Ik} and \eqref{eq:mathfrak-m}. 
\end{prop}

\subsection{Sharp blow-up analysis and asymptotics}

Let us consider the function
\begin{align}\label{def_operatore_F}
  F: \RR \times H^{1}_{0}(\Omega)  &\longrightarrow  \RR \times H^{-1}(\Omega)\\
  \notag (\lambda,\varphi) &\longmapsto \Big( {\textstyle{
      q(\varphi) -\lambda_N,
      \ -\Delta \varphi-\lambda \varphi}}\Big),
\end{align}
   where $q$ is defined in \eqref{eqQuad} and 
 $-\Delta  \varphi-\lambda \varphi\in H^{-1}(\Omega)$ acts as 
\[
\sideset{_{H^{-1}(\Omega)}^{}}{}{\mathop{\Big\langle}} -\Delta \varphi-\lambda \varphi , u
\Big\rangle_{\!H^{1}_{0}(\Omega)}\!\!=
\int_{\Omega}\nabla\varphi\cdot{\nabla u}\,dx
      -\!\lambda \!\int_{\Omega} \varphi {u} \,dx
\]
for all $\varphi\in H^{1}_{0}(\Omega)$. 
We have that $F(\lambda_N,u_N)=(0,0)$, $F$ is Fr\'{e}chet-differentiable at $(\lambda_N,u_N)$ and its Fr\'{e}chet-differential
$dF(\lambda_N,u_N)\in \mathcal L\big( \RR \times
H^{1}_{0}(\Omega),\RR \times H^{-1}(\Omega)\big)$
is invertible.
Therefore we can control $|\lambda_N(\eps)-\lambda_N|$ and
$\|w_{N,R,\eps} - u_N \|_{H^{1}_0(\Omega)}\|$ with
$\|F(\lambda_N(\eps)-w_{N,R,\eps}\|_{\RR\times
  H^{-1}(\Omega)}$. Then the norm $\|F(\lambda_N(\eps)-w_{N,R,\eps}\|_{\RR\times
  H^{-1}(\Omega)}$ can be estimated taking advantage of the
computations performed in Section \ref{sec:C3}, thus yielding 
\[
\|w_{N,R,\eps} - u_N
\|_{H^{1}_0(\Omega)}=O\Big(\sqrt{H(u_N^\eps,\mu_\delta\eps)}\Big)
\]
as $\eps\to0^+$  for every  $R>2$, $\mu_\delta$ being as in Lemma \ref{l:sec_con} for some 
$\delta\in(0,1/2)$ fixed.

As a consequence, for every $R>2$
\begin{equation}\label{eq:41}
 \int_{\big(\frac{1}{\eps}\Omega\big)\setminus D_{R}^+}\bigg|\nabla
 \Big(\hat U_\eps - 
\tfrac{\eps^{k}}{\sqrt{H(u_N^\eps,\mu_\delta\eps)}}U_\eps\Big)\bigg|^2dx=O(1),
\quad\text{as }\eps\to0^+.
\end{equation}
Using \eqref{eq:41} and the uniqueness part of Lemma
\ref{l:Phi}, we can identify univocally the limit of the blow-up
family $\{\hat U_\eps\}_\eps$ introduced in \eqref{eq:blowup} and
prove that  
\[
 \lim_{\eps\to 0^+}\frac{\eps^{k}}{\sqrt{H(u_N^\eps,\mu_\delta\eps)}}=
\frac1{|\beta|}
\sqrt{\frac{\mu_\delta}{\int_{S_{\mu_\delta}^+}|\Phi_k|^2 ds}}
\]
and 
\begin{equation}\label{eq:buU}
\hat U_\eps\to 
\frac{\beta}{|\beta|}\sqrt{\frac{\mu_\delta}{\int_{S_{\mu_\delta}^+}|\Phi_k|^2 ds}}
\, \Phi_k \quad\text{as
  }\eps\to0^+
\end{equation}
 in $H^{1 }(D_R^+)$ for every $R>1$ and
in $C^{2}_{\rm loc}(\overline{\RR^2_+}\setminus\{{\mathbf e},-{\mathbf
  e}\})$, see \cite[Theorem 8.1]{abatangelo2015sharp} for details. 

Combining \eqref{eq:buU} with the Dirichlet principle, we can prove convergence of
the blow-up family $W_\eps^R$ introduced in \eqref{eq:blowup}:
for all $R>2$,
\begin{equation}\label{eq:buW}
W_\eps^R \to  \frac{\beta}{|\beta|}
\sqrt{\frac{\mu_\delta}{\int_{S_{\mu_\delta}^+}|\Phi_k|^2ds}}\,w_R
\qquad \text{as }\eps\to0^+ \text{ in }  H^{1}(D_R^+),
\end{equation}
where $w_R$ is the unique solution to the minimization problem 
\[
  \int_{D_{R}^+} \!\!|\nabla  w_R(x)|^2\,dx
  = \min\left\{ \int_{D_{R}^+}\!\! |\nabla u|^2\,dx:\, u\in
    H^{1}(D_{R}^+), \ u= \Phi_k \text{ on }S_{R}^+,
    \ u=0 \text{ on }\Gamma_R \right\},
\]
which then  solves
\begin{equation*}
 \begin{cases}
  -\Delta w_R = 0, &\text{in }D_{R}^+,\\
  w_R = \Phi_k, &\text{on }S_{R}^+,\\
  w_R = 0, &\text{on }\Gamma_R.
   \end{cases}
\end{equation*}
To obtain the exact asymptotics for $\lambda_N- \lambda_N(\eps)$ 
it remains
to determine the limit of the function $g_R(\eps)$ defined in \eqref{eq:gR} of Proposition
\ref{p:stimadifficile} as $\eps\to 0$ and $R\to+\infty$.
 
\begin{lem}\label{l:limite_kappa_R}
For all $R>\tilde R$ and  $\eps>0$ with
  $R\eps<{R_0}$, let $g_R(\eps)$ be as in Proposition
  \ref{p:stimadifficile}.
Then 
\begin{equation}\label{eq:60}
\lim_{\eps\to0^+}g_R(\eps)= \frac{\mu_\delta}{\int_{S_{\mu_\delta}^+}|\Phi_k|^2ds} 
\tilde \kappa_R
\end{equation}
where 
\begin{equation}\label{eq:68}
\tilde \kappa_R=\int_{S_R^+}\Big(
  \nabla w_R\cdot\nu - \nabla \Phi_k\cdot\nu \Big) {\Phi_k}\,ds,
\end{equation}
with $\nu=\frac{x}{|x|}$. 
Furthermore $\lim_{R\to+\infty}\tilde \kappa_R=-2{\mathfrak m}_k$,
where ${\mathfrak m}_k$ is defined in \eqref{eq:mathfrak-m} and \eqref{eq:Ik}.
\end{lem}
\begin{proof}
We first observe that, by \eqref{eq:gR} and 
convergences \eqref{eq:buU}--\eqref{eq:buW},
\begin{align*}
  &\lim_{\eps\to0^+}g_R(\eps)= \frac{\mu_\delta}{\int_{S_{\mu_\delta}^+}|\Phi_k|^2ds}
  \bigg(\int_{D_R^+}|\nabla w_R|^2\,dx-\int_{D_R^+}|\nabla \Phi_k|^2\,dx\bigg) 
  = \frac{\mu_\delta}{\int_{S_{\mu_\delta}^+}|\Phi_k|^2ds} \tilde \kappa_R
\end{align*}
with $\tilde \kappa_R$ as in \eqref{eq:68}.
We observe that 
\begin{equation}\label{eq:109}
  \tilde \kappa_R =\int_{S_R^+}\Big( \nabla w_R\cdot\nu -
  \nabla \Phi_k\cdot\nu \Big) \psi_k\,ds+I_1(R)+I_2(R)
\end{equation}
where 
\begin{equation*}
  I_1(R)= \int_{S_R^+}\Big({\Phi_k}-\psi_k\Big) \nabla \Big(\psi_k-\Phi_k\Big) \cdot\nu
  \,ds,\quad
  I_2(R)=\int_{S_R^+}\Big({\Phi_k}-\psi_k\Big)
  \nabla \Big(w_R-\psi_k \Big) \cdot\nu  \,ds.
\end{equation*}
Testing the equation
$-\Delta \big(\psi_k-{ \Phi_k}\big) =0$ with the function 
$\psi_k-{\Phi_k}$, recalling
that $\psi_k-{ \Phi_k}=0$ on $s$, and 
integrating it over $\RR^2_+\setminus D_R^+$, thanks to Lemma \ref{l:Phi}  we obtain that
\[
 I_1(R)=  \int_{\RR^2_+\setminus D_R^+} |\nabla(\Phi_k-\psi_k)|^2=
\int_{\RR^2_+\setminus D_R^+} |\nabla w_k|^2
 \to 0 \qquad \text{as }R\to+\infty .
\]
Let $\eta_R:\RR^2_+\to\RR$ be a smooth cut-off function 
such that $\eta_R\equiv 0$ in $D_{R/2}^+$, $\eta_R\equiv 1$ in
$\RR^2_+\setminus D_R^+$, $0\leq \eta_R\leq 1$, and $|\nabla\eta_R|\leq\frac{4}{R}$.
Testing the equation $-\Delta (\psi_k - w_R)=0$ in $D_R^+$ with the function
$\eta_R({ \Phi_k}-\psi_k)$ we 
obtain 
\[
I_2(R)= \int_{D_R^+} \nabla (w_R-\psi_k) \cdot \nabla ((\Phi_k - \psi_k)\eta_R)
\]
so that, in view of  the Dirichlet Principle, Lemma \ref{l:Phi} and
the fact that $w_k\in\mathcal Q$,
\begin{align*}
 |I_2(R)| &\leq 
\int_{D_R^+} |\nabla ((\Phi_k - \psi_k)\eta_R)|^2dx\leq
2
\int_{D_R^+\setminus D_{R/2}^+} |\nabla (\Phi_k - \psi_k)|^2 \,dx+\frac{32}{R^2}
\int_{D_R^+\setminus D_{R/2}^+} |\Phi_k - \psi_k|^2 \,dx\\
&\leq 2
\int_{D_R^+\setminus D_{R/2}^+} |\nabla w_k|^2 \,dx+128
\int_{D_R^+\setminus D_{R/2}^+} \frac{w_k^2}{|x-{\mathbf e}|^2}
  \,dx\to 0
\end{align*} 
as $R\to+\infty$, where in the last line of the above estimate we have
used that $\frac{1}{R}\leq\frac2{|x-{\mathbf e}|}$ for all
$x\in D_R^+$.

  Therefore we need just to study the limit of the
quantity
$\int_{S_R^+} \frac{\partial}{\partial \nu} (w_R - \Phi_k)\psi_k$ as $R\to+\infty$.
To this aim, we consider the function
\begin{equation}\label{eq:47}
 \xi(r):= \int_{0}^\pi \Phi_k(r\cos t,r\sin t) \sin(kt)\,dt,\qquad
 r\geq 1,
\end{equation}
and notice that it satisfies the differential equation
$\xi''+\frac1r\xi'-\frac{k^2}{r^2}\xi=0$ which can be rewritten as 
$(r^{1+2k}(r^{-k}\xi)')'=0$ in $[1,+\infty)$. 
Therefore there exists some $C_\xi\in\RR$ such that 
\[
\big( r^{-k}\xi(r) \big)'= \frac{C_\xi}{r^{1+2k}}\quad\text{in }
[1,+\infty).
\]
Integrating the previous equation over $[1,r]$ we obtain that
\begin{equation}\label{eq:114} 
r^{-k} \xi(r) -\xi(1) = \frac{C_\xi}{2k}\left(1-\frac1{r^{2k}}\right). 
\end{equation}
From \eqref{eq:psi_k}, Lemma \ref{l:wk}, and Lemma \ref{l:Phi} it follows that 
\begin{align*}
  \xi(r)&=\int_0^{\pi}\psi_k(r\cos t,r\sin t)
\sin (k t)\,dt+\int_{0}^{\pi} \Big(\Phi_k(r\cos t,r\sin t)-\psi_k(r\cos t,r\sin t) \Big)
 \sin (k t)\,dt\\
&=\frac\pi2 r^{k}+O(r^{-1}),\quad\text{as }r\to+\infty,
\end{align*}
and hence 
$r^{-k} \xi(r) \to \frac\pi2$ as $r\to+\infty$. 
Letting $r\to+\infty$ in \eqref{eq:114}, this implies that $\tfrac{C_\xi}{2k}=\frac\pi2 -\xi(1)$, so that 
\begin{equation}
 \xi(r)= \tfrac\pi2 \,r^{k} +  \big( \xi(1)-\tfrac\pi2\big) r^{-k},\quad
 \xi'(r)=k \tfrac\pi2  r^{k-1} + k\big(\tfrac\pi2 - \xi(1)\big) r^{-k-1} \label{eq:xi}
\end{equation}
for all $r\geq1$. In particular, from \eqref{eq:xi} we have that 
\begin{equation*}
 \tfrac\pi2 - \xi(1) = \tfrac\pi2 r^{2k} - r^{k}\xi(r),\quad \text{for
   all }r\geq1,
\end{equation*}
whose substitution into \eqref{eq:xi} yields
\begin{equation*}
  \xi'(r)=k{\pi}r^{k-1}-k\,\frac{\xi(r)}{r},\quad \text{for
   all }r\geq1.
\end{equation*}
On the other hand, 
differentiating \eqref{eq:47} we obtain also  
\begin{equation}\label{eq:112}
  \xi'(r)=\frac1{r^{1+k}}\int_{S_r^+}
  \frac{\partial\Phi_k}{\partial\nu}\,
 \psi_k\,ds
\end{equation}
so that 
\begin{equation}\label{eq:27}
 \int_{S_r^+}
  \frac{\partial\Phi_k}{\partial\nu}\,
 \psi_k\,ds = k(\pi r^{2k} -r^k \xi(r))\quad\text{for all }r\geq1.
\end{equation}
Now we turn to 
\[
\zeta_R(r):= \int_{0}^{\pi} w_R(r\cos t,r\sin t)\,\sin(kt)\,dt
\]
which is the  $k$-th Fourier coefficient of the
harmonic function $w_R$ and hence
satisfies, for some $C_{R}\in\RR$,
$\big( r^{-k}\zeta_R(r) \big)'= \frac{C_{R}}{r^{1+2k}}$ in $(0,R]$. 
Integrating the previous equation over $[r,R]$ we obtain that  
\[
R^{-k} \zeta_R(R) -r^{-k} \zeta_R(r) =
\frac{C_{R}}{2k} \left(\frac{1}{r^{2k}}-\frac{1}{R^{2k}}\right),\quad \text{for all $r\in(0,R)$}.
\] 
By regularity of $w_R$
we necessarily have that $C_{R}=0$. Hence 
\begin{equation}
 \zeta_R(r)= \tfrac{\zeta_R(R)}{R^{k}} r^{k}\quad \text{and}\quad
 \label{eq:116}\zeta_R'(r)=k\tfrac{\zeta_R(R)}{R^{k}} r^{k-1}, \quad \text{for all }r\in(0,R].
\end{equation}
From the definition of $\zeta_R$ we have that 
$\zeta_R'(r)=\frac1{r^{1+k}}\int_{S_r^+}
  \nabla w_R\cdot\nu \,
  \psi_k\,ds$. 
Hence
\[
 \int_{S_r^+}
  \nabla w_R\cdot\nu \,
  \psi_k\,ds =k\tfrac{\zeta_R(R)}{R^k}\, r^{2k} 
\]
from which, taking into account the boundary conditions for $w_R$, it
follows that 
\begin{equation}\label{eq:37}
 \int_{S_R^+}
  \nabla w_R\cdot\nu \,
  \psi_k\,ds = kR^k \xi(R) .
\end{equation}
Combining \eqref{eq:27}, \eqref{eq:37}, and \eqref{eq:xi} we conclude
that 
\[
\int_{S_R^+} \left(\frac{\partial w_R}{\partial \nu} \psi_k -
  \frac{\partial \Phi}{\partial \nu} \psi_k \right)\,ds = 
 2kR^k \xi(R)-k\pi R^{2k}=
2k
\bigg(\xi(1) - \frac\pi2 \bigg) = - 2\mathfrak{m}_k
\]
by virtue of Lemma \ref{l:xi1}. 
\end{proof}
 
 By combining the
previous results we obtain the following asymptotics for the
eigenvalue variation.

\begin{thm}\label{t:main_asy_eige}
Let $\Omega$ be a bounded open set in $\RR^2$ satisfying 
\eqref{eq:38} and \eqref{eq:40}.
Let $N\geq 1$ be such that 
the $N$-th eigenvalue $\lambda_N$ of $q_0$ on $\Omega$  is simple 
 with associated eigenfunctions having  in
  $0$ a zero of order $k$ with $k$ as in \eqref{eq:orderk}. For
  $\eps\in (0,\eps_0)$ let $\lambda_N(\eps)$ be
  the $N$-th eigenvalue of $q_\eps$ on $\Omega$. 
Then 
\begin{equation*}
\lim_{\eps\to 0^+}\frac{\lambda_N-\lambda_N(\eps)}{\eps^{2k}}=
-2 \beta^2\,{\mathfrak m}_k
\end{equation*}
with $\beta$ being as in \eqref{eq:orderk} and
${\mathfrak m}_k$ being as in \eqref{eq:mathfrak-m} and \eqref{eq:Ik}.
\end{thm}
In particular, Theorem \ref{t:main_asy_eige} and  \eqref{eq:buU} above
provide a  proof of Theorem \ref{t:gad} that is alternative  to the
one given
in \cite{Gad}.

 \paragraph{Acknowledgements} 
The authors thank the anonymous reviewers for their
  suggestions and comments. They are also indebted to 
Andr\'e Froehly for bringing relevant references to their
attention.
 L. Abatangelo, V. Felli and C. L\'ena are partially supported by
 the project ERC Advanced Grant 2013 n. 339958: ``Complex Patterns for
 Strongly Interacting Dynamical Systems -- COMPAT''. 
L. Abatangelo and  V. Felli are partially supported by the INDAM-GNAMPA 2018 grant ``Formula di monotonia e applicazioni: problemi frazionari e stabilità spettrale rispetto a perturbazioni del dominio''.
 V. Felli is partially supported by the PRIN 2015
grant ``Variational methods, with applications to problems in
mathematical physics and geometry''. C. L\'ena is partially supported by the Portuguese FCT (Project OPTFORMA, IF/00177/2013)
and the Swedish Research Council (Grant D0497301).

\end{document}